\documentclass[11pt, twoside, leqno]{article}

\usepackage{amssymb}
\usepackage{amsmath}
\usepackage{amsthm}
\usepackage{color}
\usepackage{mathrsfs}

\allowdisplaybreaks

\def\rr{{\mathbb R}}
\def\rn{{\mathbb{R}^n}}

\def\zz{{\mathbb Z}}

\def\nn{{\mathbb N}}

\def\cg{{\mathcal G}}

\def\cs{{\mathcal S}}

\def\fz{\infty }
\def\az{\alpha}
\def\bz{\beta}

\def\gz{{\gamma}}

\def\lz{\lambda}

\def\vz{\varphi}

\def\lf{\left}
\def\r{\right}

\def\hs{\hspace{0.25cm}}
\def\ls{\lesssim}

\def\noz{\nonumber}
\def\wz{\widetilde}
\def\wh{\widehat}

\def\com{\complement}

\def\loc{{\mathop\mathrm{\,loc\,}}}
\def\supp{\mathop\mathrm{\,supp\,}}

\def\Xint#1{\mathchoice
{\XXint\displaystyle\textstyle{#1}}%
{\XXint\textstyle\scriptstyle{#1}}%
{\XXint\scriptstyle\scriptscriptstyle{#1}}%
{\XXint\scriptscriptstyle\scriptscriptstyle{#1}}%
\!\int}
\def\XXint#1#2#3{{\setbox0=\hbox{$#1{#2#3}{\int}$ }
\vcenter{\hbox{$#2#3$ }}\kern-.6\wd0}}

\def\dashint{\Xint-}

\def\ga{\gamma}

\def\F{{F}_{p,q}^\az}

\def\f{\frac}
\def\vi{\varphi}
\def\({\left(}
\def \){ \right)}

 \def\a{{\alpha}}

\def\lz{{\lambda}}

\def\al{\alpha}

 \def\supp{\operatorname{supp}}

\newcommand{\we}{\wedge}

\newtheorem{theorem}{Theorem}[section]
\newtheorem{lemma}[theorem]{Lemma}

\theoremstyle{definition}
\newtheorem{remark}[theorem]{Remark}
\newtheorem{definition}[theorem]{Definition}
\renewcommand{\appendix}{\par
   \setcounter{section}{0}%
   \setcounter{subsection}{0}%
   \setcounter{subsubsection}{0}%
   \gdef\thesection{\@Alph\c@section}%
   \gdef\thesubsection{\@Alph\c@section.\@arabic\c@subsection}%
   \gdef\theHsection{\@Alph\c@section.}%
   \gdef\theHsubsection{\@Alph\c@section.\@arabic\c@subsection}%
   \csname appendixmore\endcsname
 }

\numberwithin{equation}{section}

\begin{document}

\arraycolsep=1pt

\title{\bf\Large Littlewood-Paley Characterizations of Haj{\l}asz-Sobolev
and Triebel-Lizorkin Spaces via Averages on Balls
\footnotetext{\hspace{-0.35cm} 2010 {\it
Mathematics Subject Classification}. Primary 46E35;
Secondary 42B25, 42B35, 30L99.
\endgraf {\it Key words and phrases.} (Haj{\l}asz-)Sobolev space, Triebel-Lizorkin space,
ball average, difference, Lusin-area function, $g_\lambda^*$-function, Calder\'on reproducing formula,
space of homogeneous type.
\endgraf Der-Chen Chang is supported by the NSF grant DMS-1203845 and Hong Kong
RGC competitive earmarked research grants $\#$601813 and $\#$601410.
This project is also supported by the National
Natural Science Foundation of China
(Grant Nos.~11171027, 11361020 and 11471042),
the Specialized Research Fund for the Doctoral Program of Higher Education
of China (Grant No. 20120003110003) and the Fundamental Research
Funds for Central Universities of China (Grant Nos. 2013YB60
and 2014KJJCA10).
}}
\author{Der-Chen Chang, Jun Liu, Dachun Yang\,\footnote{Corresponding author}\ \
and Wen Yuan}
\date{}
\maketitle

\vspace{-0.7cm}

\begin{center}
\begin{minipage}{13cm}
{\small {\bf Abstract}\quad Let $p\in(1,\infty)$ and $q\in[1,\infty)$. In this article,
the authors characterize the Triebel-Lizorkin space
${F}^\alpha_{p,q}(\mathbb{R}^n)$ with  smoothness order $\alpha\in(0,2)$ via the Lusin-area function and the $g_\lambda^*$-function in terms of
difference between $f(x)$ and its average
$B_tf(x):=\frac1{|B(x,t)|}\int_{B(x,t)}f(y)\,dy$ over a ball
$B(x,t)$ centered at $x\in\mathbb{R}^n$ with radius $t\in(0,1)$. As an application, the authors obtain a series of
characterizations of $F^\alpha_{p,\infty}(\mathbb{R}^n)$ via pointwise inequalities, involving ball averages, in spirit
close to Haj{\l}asz gradients, here an interesting phenomena naturally appears that,
in the end-point case when $\alpha =2$, these pointwise inequalities
characterize the Triebel-Lizorkin spaces $F^2_{p,2}(\mathbb{R}^n)$, while not
$F^2_{p,\infty}(\mathbb{R}^n)$. In particular, some new pointwise
characterizations of Haj{\l}asz-Sobolev spaces via ball averages are obtained.
Since these new characterizations only use ball averages,
they can be used as starting points for developing a theory of
Triebel-Lizorkin spaces with smoothness orders not less than
$1$ on spaces of homogeneous type.
}
\end{minipage}
\end{center}

\vspace{0.2cm}

\section{Introduction}\label{s1}
\hskip\parindent
The theory of function spaces with smoothness is one of central
topics of analysis on metric measure spaces.
In 1996, Haj{\l}asz \cite{h96} introduced the notion of Haj{\l}asz gradients,
which serves as a powerful tool to develop the first order Sobolev spaces on metric
measure spaces. Later
Shanmugalingam \cite{sh00} introduced another kind of the first order
Sobolev space by means of upper gradients.
Via introducing the fractional version of Haj{\l}asz gradients,
Hu \cite{hu03} and Yang \cite{y03} introduced
Sobolev spaces with smoothness order $\al\in(0,1)$ on fractals and
metric measure spaces, respectively. However, how to introduce a suitable and
useful Sobolev space with smoothness order bigger than 1 on metric measure spaces
is still an open problem. Due to the lack
of differential structures on metric measure spaces,
one key step to solve the above problem is to find some suitable substitute of the usual
high order derivatives on metric measures spaces.

Via a pointwise inequality involving the higher order differences,
Triebel \cite{t10,t11} and Haroske and Triebel
\cite{ht11,ht13} obtained some pointwise characterizations, in the spirit of Haj{\l}asz \cite{h96} (see also Hu \cite{h03} and Yang \cite{y03}), of Sobolev spaces on $\rn$
with smoothness order bigger than $1$.
However, it is still unclear how to introduce higher than $1$ order differences on
spaces of homogeneous type. Notice also that, in \cite{llw02}, under \emph{a priori assumption}
on the existence of polynomials, Liu et al. introduced the Sobolev spaces
of higher order on metric measure spaces.
Recently, Alabern et al. \cite{amv}
obtained a new interesting characterization of Sobolev spaces with smoothness order bigger than $1$
on $\rn$ via ball averages, which provides a possible way to introduce higher order Sobolev spaces on metric measure spaces.
The corresponding characterizations for Besov and Triebel-Lizorkin spaces were later considered by Yang et al.   \cite{yyz}.

Via differences involving ball averages, Dai et al. \cite{dgyy} provides several other ways,
which are different from \cite{amv} and in spirit more close to the pointwise
characterization as in \cite{h96,h03,y03}, to
introduce Sobolev spaces of order $2\ell$
on spaces of homogeneous type in the sense of Coifman and Weiss \cite{cw71,cw77}, where $\ell\in\nn:=\{1,2,\ldots\}$.
Moreover, Dai et al. \cite{dgyy0} further characterized Besov and Triebel-Lizorkin
spaces with smoothness order in $(0, 2\ell)$ via differences involving ball averages, which also gave out
a possible way to introduce Besov and Triebel-Lizorkin spaces with any positive smoothness order
on spaces of homogeneous type. In particular, when $\az\in(0,2)$, $p\in(1,\fz)$
and $q\in(1,\fz]$, it was proved in \cite[Theorem 3.1(ii)]{dgyy0} that a locally integrable function $f$ belongs to the  Triebel-Lizorkin space $F^\az_{p,q}(\rn)$
if and only if
\begin{equation}\label{eq1}
\|f\|_{L^p(\rn)}+\lf\|\lf\{\sum_{k=0}^\fz2^{k\az q} |f-B_{2^{-k}}f|^q\r\}^{1/q}\r\|_{L^p(\rn)}<\fz;
\end{equation}
moreover, the quantity in \eqref{eq1} is an equivalent quasi-norm of
$F^\az_{p,q}(\rn)$. Here and hereafter, for any locally integrable function $f$,
$t\in(0,\fz)$ and $x\in\rn$, we let
\begin{equation*}
B_tf(x):= \frac1{|B(x,t)|}\int_{B(x,t)} f(y) \,dy=:\dashint_{B(x,t)} f(y) \,dy,
\end{equation*} and
$B(x,t)$ stand for a ball centered at $x$ with radius $t$.
Observe that this result in \cite[Theorem 3.1(ii)]{dgyy0} can be regarded as the
characterization of $F^\az_{p,q}(\rn)$
via a Littlewood-Paley $\cg$-function involving $f-B_{2^{-k}}f$.
The corresponding result for homogeneous Triebel-Lizorkin spaces was also obtained
in \cite{dgyy0}.

The main purpose of this article is to establish some Lusin-area function and $g_\lambda^*$-function variants of  the above
characterization for Triebel-Lizorkin
spaces ${F}^\alpha_{p,q}(\mathbb{R}^n)$,
which also provide some other possible ways to introduce
Triebel-Lizorkin spaces with smoothness orders not less than $1$ on spaces of homogeneous type.
As an application, we obtain a series of
characterizations of $F^\alpha_{p,\infty}(\mathbb{R}^n)$ for $\az\in (0,2)$ and
$p\in (1,\fz)$ via pointwise inequalities, involving ball averages, in spirit
close to Haj{\l}asz gradients, here an interesting phenomena naturally appears that,
in the end-point case when $\alpha =2$, these pointwise inequalities
characterize the Triebel-Lizorkin spaces $F^2_{p,2}(\mathbb{R}^n)$, while not
$F^2_{p,\infty}(\mathbb{R}^n)$.
Recall that, for $p\in(1,\fz)$, the Haj{\l}asz-Sobolev spaces $M^{\az,p}(\rn)$
coincide with $F^1_{p,2}(\rn)$ when $\az=1$ and with $F^\az_{p,\fz}(\rn)$
when $\az\in(0,1)$ (see \cite{h96,y03} and also Remark \ref{re}(i) below). Thus,
these pointwise characterizations also lead to some new pointwise characterizations of (fractional) Haj{\l}asz-Sobolev
spaces in spirit of \cite{dgyy}, which are different from those obtained in \cite{h96,h03,hu03,y03}.
Recall that the pointwise characterizations of Besov and Triebel-Lizorkin spaces play
important and key roles in the study for the invariance of these function spaces under quasi-conformal mappings;
see, for example, \cite{kyz11,gkz13,hk13,kkss14,bss14}.

To state our main results of this article, we first recall some basic notions. Denote by
$L^1_\loc(\rn)$ the collection of all
locally integrable functions  on $\rn$.
Let  $\mathcal{S}(\rn)$ denote the collection of all
\emph{Schwartz functions} on $\rn$, endowed
with the usual topology, and $\mathcal{S}'(\rn)$ its \emph{topological dual}, namely,
the collection of all bounded linear functionals on $\mathcal{S}(\rn)$
endowed with the weak $\ast$-topology.
Let $\zz_+:=\{0,1,\ldots\}$ and  $\cs_\fz(\rn)$ be the set of all Schwartz functions $\vz$ such that
$\int_\rn x^\gz \vz(x)\,dx=0$ for all $\gz\in \zz_+^n$, and
$\cs'_\fz(\rn)$ its topological dual.
For all $\az\in\zz_+^n$, $m\in\zz_+$ and $\vz\in\cs(\rn)$, let
$$\|\vz\|_{\az,m}:=\sup_{|\bz|\le|\az|,\,x\in\rn}(1+|x|)^m |\partial^\bz \vz(x)|.$$

We also use  $\widehat{\varphi}=\varphi^{\we}$ and $\varphi^{\vee}$ to denote
the \emph{Fourier transform} and the \emph{inverse transform} of $\varphi$,
respectively. For any $\vz\in \cs(\rn)$ and $t\in(0,\fz)$, we let $\vz_t(\cdot):=t^{-n}\vz(\cdot/t)$.
For any $E\subset\rn$, let $\chi_E$ be its \emph{characteristic function}.

The Triebel-Lizorkin spaces are defined as follows (see \cite{t83,t92,FJ90,ysy}).

\begin{definition}\label{d2.1}
Let $\az\in(0,\,\fz)$, $p,\,q\in(0,\,\fz]$, $\vz,\,\Phi\in\cs(\rn)$ satisfy that
\begin{equation}\label{con}
\supp\,\widehat\vz\subset \{\xi\in\rn:\ 1/2\le|\xi|\le2\}\ \mathrm{and}\
|\widehat\vz(\xi)|\ge {\rm constant}>0 \ \mathrm{if}\ 3/5\le|\xi|\le5/3
\end{equation}
and
\begin{equation}\label{con1}
\supp\,\widehat\Phi\subset \{\xi\in\rn:\ |\xi|\le2\}\ \mathrm{and}\
|\widehat\Phi(\xi)|\ge {\rm constant}>0 \ \mathrm{if}\ |\xi|\le5/3.
\end{equation}
The {\it Triebel-Lizorkin space}
$F^\az_{p,\,q}(\rn)$ is defined as the collection of all
$f\in\cs'(\rn)$ such that $\|f\|_{F^\az_{p,\,q}(\rn)}<\fz$,
where, when $p\in(0,\fz)$,
\begin{equation*}
\|f\|_{F^\az_{p,\,q}(\rn)}:=\lf\|\lf[\sum_{k=0}^\fz2^{k\az q}
  |\vz_{2^{-k}}\ast
f|^q\r]^{1/q}\r\|_{L^p(\rn)},
\end{equation*}
and
\begin{equation*}
\|f\|_{F^\az_{\fz,\,q}(\rn)}:=\sup_{x\in\rn}
\sup_{m\in\zz_+}\lf\{\dashint_{B(x,\,2^{-m})}\sum_{k=m}^\fz2^{k\az q}
 |\vz_{2^{-k}}\ast
f(y)|^q\,dy\r\}^{1/q}
\end{equation*}
with $\vz_{2^{-k}}$ when $k=0$ replaced by $\Phi$ and
the usual modification made when $q=\fz$.
\end{definition}

\begin{remark}\label{sr1}
(i) It is well known that the space $F^\az_{p,q}(\rn)$ is independent of the choice of
the pair $(\vz,\Phi)$ satisfying \eqref{con} and \eqref{con1}.

(ii) Let $\Phi$ and $\vz$ be as in Definition \ref{d2.1}.
It is well known that, if
$p\in(0,\,\fz),\,q\in(0,\,\fz]$ and $\az\in(0,\,\fz)$, then, for all $f\in\cs'(\rn)$,
\begin{eqnarray*}
\|f\|_{F^\az_{p,\,q}(\rn)} &&\sim \|\Phi\ast f\|_{L^p(\rn)}+\lf\|\lf[\int_0^1 t^{-\az q}
|\vz_t\ast
f|^q\,\frac{dt}t\r]^{1/q}\r\|_{L^p(\rn)}\\
&&\sim  \|\Phi\ast f\|_{L^p(\rn)}+\lf\|\lf[\int_0^1 t^{-\az q}
\dashint_{B(\cdot,\,t)}\lf|\vz_t\ast
f(y)\r|^q\,dy\,\frac{dt}t\r]^{1/q}\r\|_{L^p(\rn)}\\
 &&\sim  \|\Phi\ast f\|_{L^p(\rn)}+\lf\|\lf\{\int_0^1 t^{-\az q}
\lf[\dashint_{B(\cdot,\,t)}\lf|\vz_t\ast
f(y)\r|\,dy\r]^q\,\frac{dt}t\r\}^{1/q}\r\|_{L^p(\rn)}
\end{eqnarray*}
with equivalent positive constants being independent of $f$;
see, for example, \cite{lsuyy,u12}. Indeed,
the first and the second equivalences can be found in
\cite[Theorem 2.6]{u12}, and the third one follows from a slight modification of the proof of
\cite[Theorem 2.6]{u12}, the details being omitted.

(iii) It is known that, when $p\in(0,\fz)$, $q\in(0,\fz]$ and $\az\in( n\max\{0,1/p-1/q\},1)$,
then $f\in F^\az_{p,q}(\rn)$ if and only if
$$\|f\|_{L^p(\rn)}+\lf\|\lf[\int_0^1 t^{-\az q}
\dashint_{B(\cdot,\,t)}\lf|f(\cdot)-
f(y)\r|^q\,dy\,\frac{dt}t\r]^{1/q}\r\|_{L^p(\rn)}<\fz,$$
which also serves as an equivalent quasi-norm of $F^\az_{p,q}(\rn)$; see
\cite[Section 3.5.3]{t92}.
\end{remark}

The following result is a slight variant of the `continuous'
version of \cite[Theorem 3.1(ii)]{dgyy0} when $p\in(0,\fz)$
and $\ell=1$.

\begin{theorem}\label{t-bf}
Let $p\in(1,\fz),\,q\in(1,\infty]$ and $\a\in (0, 2)$.
Then $f\in F^\az_{p,q}(\rn)$ if and only if $f\in L^p(\rn)$ and
\begin{eqnarray*}
|||f|||_{\F(\rn)}:=\|f\|_{L^p(\rn)}+\left\|\lf[\int_0^1 t^{-\az q}
|f-B_tf|^q\,\frac{dt}t\r]^{1/q}\right\|_{L^p(\rn)}<\fz.
\end{eqnarray*}
Moreover, $|||\cdot|||_{F^\az_{p,q}(\rn)}$ is an equivalent norm of
$F^\az_{p,q}(\rn)$.
\end{theorem}

Recall that the \emph{Hardy-Littlewood maximal operator $M$} is defined by
setting, for all $f\in L^1_\loc(\rn)$,
$$Mf(x):=\sup_{x\in B} \dashint_B |f(y)|\,dy,\quad x\in\rn,$$
where the supremum is taken over all balls $B$ in $\rn$ containing $x$.

\begin{remark}\label{r-a1} Let all the notation be the same as in Theorem \ref{t-bf}.
Then, from the boundedness of the Hardy-Littlewood maximal
function $M$ on $L^p(\rn)$ with $p\in (1,\fz)$,
it is easy to deduce that there exists a positive constant $C$ such that, for
all $p\in(1,\fz),\,q\in(1,\infty]$, $\a\in (0, 2)$ and $f\in L^p(\rn)$,
$$\left\|\lf[\int_1^\fz t^{-\az q}
|f-B_tf|^q\,\frac{dt}t\r]^{1/q}\right\|_{L^p(\rn)}\le C\|Mf\|_{L^p(\rn)}
\le C\|f\|_{L^p(\rn)}.$$
By this, we conclude that
\begin{eqnarray*}
|||f|||_{\F(\rn)}\sim\|f\|_{L^p(\rn)}+\left\|\lf[\int_0^\fz t^{-\az q}
|f-B_tf|^q\,\frac{dt}t\r]^{1/q}\right\|_{L^p(\rn)}
\end{eqnarray*}
with the equivalent positive constants independent of $f$.
This further indicates that Theorem \ref{t-bf} is a natural generalization
of \cite[Theorem 1]{amv} in the sense that \cite[Theorem 1]{amv} coincides with
Theorem \ref{t-bf} in the case $\az=1$ and $q=2$.
We also point out that the method used to show Theorem \ref{t-bf} is similar to
the proof of \cite[Theorem 3.1(ii)]{dgyy0},
but totally different from the proof of \cite[Theorem 1]{amv}.
\end{remark}

The main results of this article are the following characterizations of $\F(\rn)$
via Lusin-area functions (Theorems \ref{apcF2} and \ref{apcF0})
and $g_\lambda^*$- functions (Theorem \ref{apcF1}).

\begin{theorem}\label{apcF2}
Let $p\in(1,\fz),\,q\in(1,\infty],\,r\in[1,q)$ and $\a\in (0,2)$.
Then the following statements are equivalent:

{\rm(i)} $f\in \F(\rn)$;

{\rm(ii)} $f\in L^p(\rn)$ and
\begin{eqnarray*}
|||f|||^{(r)}_{\F(\rn)}&&:=\|f\|_{L^p(\rn)}\\
&&\quad+\lf\|\lf\{\int_0^1 t^{-\az q}
\lf[\dashint_{B(\cdot,\,t)}\lf|f(y)-B_tf(y)\r|^r\,dy\r]^{\frac qr}\,\frac{dt}t\r\}^{1/q}\r\|_{L^p(\rn)}<\fz.\nonumber
\end{eqnarray*}

Moreover, $|||\cdot|||^{(r)}_{F^\az_{p,q}(\rn)}$ is an equivalent norm of
$F^\az_{p,q}(\rn)$.
\end{theorem}

For the case $r=q$, we have the following conclusions.

\begin{theorem}\label{apcF0}
Let $p\in(1,\fz)$ and $q\in(1,\infty]$.

{\rm (i)} If $p\in[q,\fz)$ and
$\a\in (0,2)$, or $p\in(1,q)$ and
$\a\in (n(1/p-1/q),1)$, then $f\in \F(\rn)$ implies that $f\in L^p(\rn)$ and
\begin{eqnarray*}
\widetilde{|||f|||}_{\F(\rn)}:=\|f\|_{L^p(\rn)}+\lf\|\lf[\int_0^1 t^{-\az q}
\dashint_{B(\cdot,\,t)}\lf|f(y)-B_tf(y)\r|^q\,dy\,\frac{dt}t\r]^{1/q}\r\|_{L^p(\rn)}
\end{eqnarray*}
is controlled by $\|f\|_{\F(\rn)}$ modulus a positive constant independent  of $f$.

{\rm (ii)} If $\a\in (0,2)$, then $f\in L^p(\rn)$ and
$\widetilde{|||f|||}_{\F(\rn)}<\fz$ imply that $f\in \F(\rn)$ and $\|f\|_{\F(\rn)}
\le C\widetilde{|||f|||}_{\F(\rn)}$ for some positive constant $C$ independent of $f$.
\end{theorem}

\begin{remark}\label{R1.5x}
We point out that the ball averages $\dashint_{B(\cdot,t)}$ in Theorems \ref{apcF2} and \ref{apcF0}
can be replaced by $\dashint_{B(\cdot,\widetilde{C}t)}$ for any fixed positive constant $\widetilde{C}$.
\end{remark}

\begin{theorem}\label{apcF1}
Let $p,\,q\in(1,\infty)$.

{\rm (i)} If $p\in[q,\fz)$ and
$\a\in (0,2)$, or $p\in(1,q)$ and
$\a\in (n(1/p-1/q),1)$, then $f\in \F(\rn)$ implies that $f\in L^p(\rn)$ and
\begin{eqnarray*}
\overline{|||f|||}_{\F(\rn)}
&&:=\|f\|_{L^p(\rn)}\\
&&\quad+\lf\|\lf[\int_0^1 t^{-\az q}\int_{\mathbb  R^n}\lf(\frac{t}{t+|\cdot-y|}\r)^{\lambda
 n}\,|f(y)-B_tf(y)|^q\,\frac{dy\,dt}{t^{n+1}}\r]^{1/q}\r\|_{L^p(\rn)}
\end{eqnarray*}
is controlled by $\|f\|_{\F(\rn)}$ modulus a positive constant independent  of $f$,
where
$$\lz\in(q/\min\{q,p\},\fz).$$

{\rm (ii)} If $\a\in (0,2)$ and $\lz\in(1,\fz)$, then $f\in L^p(\rn)$ and
$\overline{|||f|||}_{\F(\rn)}<\fz$ imply that $f\in \F(\rn)$ and $\|f\|_{\F(\rn)}
\le C\overline{|||f|||}_{\F(\rn)}$ for some positive constant $C$ independent of $f$.
\end{theorem}

\begin{remark}\label{R1.7x}
Observe that there exists a restriction $\az\in (n(1/p-1/q),1)$
in Theorems \ref{apcF0}(i) and \ref{apcF1}(i) when $p\in(1,q)$.
This restriction comes from an application of the Lusin-area characterization
of $\F(\rn)$ involving the first order difference
(see Remark \ref{sr1}(iii)) in the
proofs of Theorems \ref{apcF0}(i) and \ref{apcF1}(i).
We believe that $n(1/p-1/q)$ might be a reasonable
lower bound of $\az$ in Theorems \ref{apcF0}(i) and \ref{apcF1}(i).
However, since we use $f-B_tf$ instead of the forward first order difference
in these two theorems, it might be possible that
Theorems \ref{apcF0}(i) and \ref{apcF1}(i)
remain true when $p\in(1,q)$ and $\az\in [1,2)$, which is still unclear
so far.
\end{remark}

By applying Theorems \ref{t-bf} and \ref{apcF0},  we obtain the following pointwise characterizations of the space
$F^{\az}_{p,\,\fz}(\rn)$ with $\az\in(0,2)$ and $p\in(1,\fz)$ via the average operator $B_t$ in spirit close
to Haj{\l}asz gradients.

\begin{theorem}\label{apcF}
Let $\az\in(0,2)$ and $p\in(1,\fz)$.
Then the following statements are equivalent:

{\rm (i)} $f\in F^{\az}_{p,\,\fz}(\rn)$;

{\rm (ii)} $f\in L^p(\rn)$ and there exist a non-negative $g\in L^p(\rn)$ and a positive constant $C_0$ such that, for all
$t\in(0,1)$ and almost every $x\in\rn$,
\begin{eqnarray*}
\lf|f(x)-B_tf(x)\r|\leq C_0t^{\az}g(x).
\end{eqnarray*}
Moreover, if $\al\in(n/p,1)$, then either of (i) and (ii) is also equivalent to the following:

{\rm (iii)} $f\in L^p(\rn)$ and there exist a non-negative $g\in L^p(\rn)$ and positive constants $C_1,\,C_2$
such that, for all
$t\in(0,1)$ and almost every $x\in\rn$ and $y\in B(x,\,C_1t)$,
\begin{eqnarray*}
\lf|f(x)-B_tf(x)\r|\leq C_2t^{\az}g(y).
\end{eqnarray*}

In any one of the above cases, the function $g$ can be chosen so that $\|f\|_{L^p(\rn)}+\|g\|_{L^p(\rn)}$ is equivalent to $\|f\|_{F^\az_{p,\,\fz}(\rn)}$ with equivalent positive constants independent of $f$.
\end{theorem}

The characterizations of $F^\az_{p,\,\fz}(\rn)$ in (ii) and (iii) of Theorem \ref{apcF} have several
interesting variants, which are stated as follows.

\begin{theorem}\label{t1-x}
Let $p\in(1,\fz)$.\\
 {\rm a)} If $\al\in(n/p,1)$, then the following statements are equivalent:

{\rm(i)} $f\in F^\az_{p,\,\fz}(\rn)$;

{\rm(ii)} $f\in L^p(\rn)$ and there exist a non-negative $g\in L^p(\rn)$ and positive constants
$C_3$, $C_4$ such that, for all $t\in(0,1)$ and almost every $x\in\rn$,
\begin{equation*}
|f(x)-B_t f(x)|\le C_3 t^\al \dashint_{B(x,C_4t)} g(y)\,dy;
\end{equation*}

{\rm(iii)} $f\in L^p(\rn)$ and there exist $q\in[1,p)$, a non-negative $g\in L^p(\rn)$ and positive constants
$C_5$, $C_6$ such that, for all $t\in(0,1)$ and almost every $x\in\rn$,
\begin{equation*}
|f(x)-B_t f(x)|\le C_5 t^\al \lf\{\dashint_{B(x,C_6t)} [g(y)]^q\,dy\r\}^{1/q}.
\end{equation*}

In any one of the above cases, the function $g$ can be chosen so that $\|f\|_{L^p(\rn)}+ \|g\|_{L^p(\rn)}$ is equivalent to $\|f\|_{F^\az_{p,\,\fz}(\rn)}$ with equivalent positive constants independent of $f$.\\
{\rm b)} If $\al\in(0,2)$, then  (ii) or (iii) in a) implies (i).
\end{theorem}

We also have some integral variants of (ii) and (iii) in Theorem \ref{apcF} as follows.

\begin{theorem}\label{t1-xx}
Let $p\in(1,\fz)$. \\
{\rm a)} If $\al\in(n/p,1)$, then the following statements are equivalent:

{\rm(i)} $f\in F^\az_{p,\,\fz}(\rn)$;

{\rm(ii)} $f\in L^p(\rn)$ and  there exist a non-negative $g\in L^p(\rn)$ and positive constants
$C_7$, $C_8$, $C_9$ such that, for all $t\in(0,1)$ and almost every $x\in\rn$,
\begin{eqnarray*}
\sup_{y\in B(x,t)}\lf|f(y)-B_{C_7t} f(y)\r|\le C_8 t^\al \dashint_{B(x,C_9t)} g(y)\,dy;
\end{eqnarray*}

{\rm(iii)} $f\in L^p(\rn)$ and there exist a non-negative $g\in L^p(\rn)$ and positive constants
$C_{10}$, $C_{11}$, $C_{12}$ such that, for all $t\in(0,1)$ and almost every $x\in\rn$,
\begin{equation*}
\dashint_{B(x,t)} |f(y)-B_{C_{10}t} f(y)|\,dy\le C_{11} t^\al \dashint_{B(x,C_{12}t)} g(y)\,dy;
\end{equation*}

{\rm(iv)} $f\in L^p(\rn)$ and there exist $r\in[1,\fz)$, a non-negative $g\in L^p(\rn)$ and positive constants
$C_{13}$, $C_{14}$, $C_{15}$ such that, for all $t\in(0,1)$ and almost every $x\in\rn$,
\begin{equation*}
\lf[\dashint_{B(x,t)}\lf|f(y)-B_{C_{13}t} f(y)\r|^r\,dy\r]^{\frac 1r}\le C_{14} t^\al \dashint_{B(x,C_{15}t)} g(y)\,dy;
\end{equation*}

{\rm(v)} $f\in L^p(\rn)$ and there exist $q\in[1,p)$,
a non-negative $g\in L^p(\rn)$ and positive constants
$C_{16}$, $C_{17}$, $C_{18}$ such that, for all $t\in(0,1)$ and almost every $x\in\rn$,
\begin{equation*}
\sup_{y\in B(x,t)}\lf|f(y)-B_{C_{16}t} f(y)\r|\le C_{17} t^\al \lf\{\dashint_{B(x,C_{18}t)} [g(y)]^q\,dy\r\}^{1/q};
\end{equation*}

{\rm(vi)} $f\in L^p(\rn)$ and there exist $q\in[1,p)$,
a non-negative $g\in L^p(\rn)$ and positive constants
$C_{19}$, $C_{20}$, $C_{21}$ such that, for all $t\in(0,1)$ and almost every $x\in\rn$,
\begin{equation*}
\dashint_{B(x,t)} |f(y)-B_{C_{19}t} f(y)|\,dy\le C_{20} t^\al \lf\{\dashint_{B(x,C_{21}t)} [g(y)]^q\,dy\r\}^{1/q};
\end{equation*}

{\rm(vii)} $f\in L^p(\rn)$ and there exist  $r\in[1,\fz) $, $q\in[1,p)$,
a non-negative $g\in L^p(\rn)$ and positive constants
$C_{22}$, $C_{23}$, $C_{24}$ such that, for all $t\in(0,1)$ and almost every $x\in\rn$,
\begin{equation*}
\lf[\dashint_{B(x,t)}\lf|f(y)-B_{C_{22}t} f(y)\r|^r\,dy\r]^{\frac 1r}\le C_{23} t^\al \lf\{\dashint_{B(x,C_{24}t)} [g(y)]^q\,dy\r\}^{1/q}.
\end{equation*}

In any one of the above cases, the function $g$ can be chosen so that $\|f\|_{L^p(\rn)}+ \|g\|_{L^p(\rn)}$ is equivalent to $\|f\|_{F^\az_{p,\,\fz}(\rn)}$ with equivalent positive constants independent of $f$.\\
{\rm b)} If $\al\in(0,2)$, then any one of the above statements (ii) through (vii) in a) implies (i).
\end{theorem}

\begin{theorem}\label{t1-xxx}
Let $p\in(1,\fz)$. \\
{\rm a)} If $\al\in(n/p,1)$,  then the following statements are equivalent:

{\rm(i)} $f\in F^\az_{p,\,\fz}(\rn)$;

{\rm(ii)} $f\in L^p(\rn)$ and there exist a non-negative $g\in L^p(\rn)$ and positive constants
$C_{25}$, $C_{26}$ such that, for all $t\in(0,1)$ and almost every $x\in\rn$,
\begin{eqnarray*}
\sup_{y\in B(x,t)}\lf|f(y)-B_{C_{25}t} f(y)\r|\le C_{26} t^\al  g(x);
\end{eqnarray*}

{\rm(iii)} $f\in L^p(\rn)$ and there exist $r\in[1,\fz)$, a non-negative $g\in L^p(\rn)$ and positive constants
$C_{27}$, $C_{28}$ such that, for all $t\in(0,1)$ and almost every $x\in\rn$,
\begin{equation*}
\lf[\dashint_{B(x,t)}\lf|f(y)-B_{C_{27}t} f(y)\r|^r\,dy\r]^{\frac 1r}\le C_{28} t^\al  g(x);
\end{equation*}

{\rm(iv)} $f\in L^p(\rn)$ and  there exist a non-negative $g\in L^p(\rn)$ and positive constants
$C_{29}$, $C_{30}$ such that, for all $t\in(0,1)$ and almost every $x\in\rn$,
\begin{equation*}
\dashint_{B(x,t)} |f(y)-B_{C_{29}t} f(y)|\,dy\le C_{30} t^\al  g(x);
\end{equation*}

{\rm(v)} $f\in L^p(\rn)$ and there exist
a non-negative $g\in L^p(\rn)$ and positive constants
$C_{31}$, $C_{32}$, $C_{33}$ such that, for all $t\in(0,1)$ and almost every $x\in\rn$ and
$y\in B(x,\,C_{31}t)$,
\begin{equation*}
\sup_{y\in B(x,t)}\lf|f(y)-B_{C_{32}t} f(y)\r|\le C_{33} t^\al g(y);
\end{equation*}

{\rm(vi)} $f\in L^p(\rn)$ and there exist $r\in[1,\fz)$,
a non-negative $g\in L^p(\rn)$ and positive constants
$C_{34}$, $C_{35}$, $C_{36}$ such that, for all $t\in(0,1)$ and almost every $x\in\rn$ and
$y\in B(x,\,C_{34}t)$,
\begin{equation*}
\lf[\dashint_{B(x,t)}\lf|f(y)-B_{C_{35}t} f(y)\r|^r\,dy\r]^{\frac 1r}\le C_{36} t^\al g(y);
\end{equation*}

{\rm(vii)} $f\in L^p(\rn)$ and there exist
a non-negative $g\in L^p(\rn)$ and positive constants
$C_{37}$, $C_{38}$, $C_{39}$ such that, for all $t\in(0,1)$ and almost every $x\in\rn$ and
$y\in B(x,\,C_{37}t)$,
\begin{equation*}
\dashint_{B(x,t)} |f(y)-B_{C_{38}t} f(y)|\,dy\le C_{39} t^\al g(y).
\end{equation*}

In any one of the above cases, the function $g$ can be chosen so that $\|f\|_{L^p(\rn)}+ \|g\|_{L^p(\rn)}$ is equivalent to $\|f\|_{F^\az_{p,\,\fz}(\rn)}$ with equivalent positive constants independent of $f$.\\
{\rm b)} If $\al\in(0,2)$, then any one of the above statements (ii) through (vii) in a) implies (i). Moreover,
the statements (i), (iii) and (iv) in a) are equivalent for $\al\in(0,2)$.
\end{theorem}

\begin{remark}\label{re}
(i) Recall that, by \cite[Corollary 1.3]{y03}, \cite[Corollary 1.2]{kyz10} and
\cite[Proposition 2.1]{kyz11} (see also \cite[Remark 3.3(ii)]{kyz11}),
for $\az\in(0,1)$ and $p\in(\frac{n}{n+\az},\fz)$, the Triebel-Lizorkin space
$F^\az_{p,\fz}(\rn)$ coincides with the \emph{fractional
Haj{\l}asz-Sobolev space $M^{\az,p}(\rn)$},
which is defined in \cite{y03} as the collection of all functions $f\in L^p(\rn)$
such that there exist a nonnegative function $g\in L^p(\rn)$ and
$E\subset \rn$ with measure zero so that
\begin{equation}\label{h-g}
|f(x)-f(y)|\le |x-y|^\az [g(x)+g(y)],\quad x,\,y\in \rn\setminus E.
\end{equation}
Such function $g$ is called the \emph{$\az$-fractional Haj{\l}asz gradient} of $f$.
The quasi-norm of $f$ in $M^{\az,p}(\rn)$ is then given by
$\|f\|_{L^p(\rn)}+\inf\{\|g\|_{L^p(\rn)}\}$, where the infimum is taken over all
such $\az$-fractional Haj{\l}asz gradients of $f$.

By the above equivalence, we see that Theorems \ref{apcF} through
\ref{t1-xxx} provide some new pointwise characterizations of fractional
Haj{\l}asz-Sobolev spaces $M^{\az,p}(\rn)$ via the differences between $f$ and its ball
average $B_tf$, which is different from the well-known
pointwise characterization of $M^{\az,p}(\rn)$ via Haj{\l}asz gradients as in \eqref{h-g}.

(ii) It was proved in \cite{dgyy} that a locally integrable function $f$ belongs
to Sobolev space ${W}^{2,p}(\rn)$, with $p\in(1,\fz)$, if and only if  either of (ii) and (iii) of
Theorems \ref{apcF} and \ref{t1-x}, or one of (ii) through (vii) of Theorems \ref{t1-xx}
and \ref{t1-xxx} holds true with $\al=2$. Notice that
${W}^{2,p}(\rn)=F^2_{p,\,2}(\rn)$ for all $p\in(1,\fz)$.
Comparing Theorems \ref{apcF} through
\ref{t1-xxx} with \cite[Theorems 1.1 through 1.4]{dgyy}, we find a
\emph{jump} of the parameter $q$ of Triebel-Lizorkin spaces $F^\az_{p, q}(\rn)$ when $\al=2$
and $\az\in (0,2)$ for the
above pointwise characterizations. More precise, letting $p\in(1,\fz)$,
any one of the items of Theorem \ref{apcF}(ii), and  (iii) and (iv) of
Theorem \ref{t1-xxx} when $\az\in(0,2)$ characterize $F^\az_{p,\,\fz}(\rn)$, while, when $\az=2$,
they characterize $F^2_{p,\,2}(\rn)$.
This interesting phenomena also appears in the pointwise characterizations
of Triebel-Lizorkin spaces $F^\az_{p, q}(\rn)$ via Haj{\l}asz gradients with
$p\in (1,\fz)$ and $q\in \{2,\fz\}$, but $\az\in (0,1]$ (see \cite{y03}).

(iii) We point out that the discrete versions of Theorems \ref{apcF2} through
\ref{t1-xxx}, namely, the conclusions via replacing $t$ by $2^{-k}$ and
$\int_0^1\cdots \frac{dt}t$ by $\sum_{k\in\zz_+}$ in those statements of
Theorems \ref{apcF2} through \ref{t1-xxx}, are also true.

(iv) In view of (i) of this Remark, the pointwise characterizations in
Theorems \ref{apcF} through \ref{t1-xxx} provide some possible ways to
introduce (fractional) Sobolev spaces with smoothness in $(0,2)$ on metric measure spaces.
Indeed, we can prove that some statements of Theorems \ref{apcF} through \ref{t1-xxx}
are still equivalent on spaces of homogeneous type in Subsection \ref{s3.2} below.
\end{remark}

The proofs of Theorems \ref{t-bf}, \ref{apcF2}, \ref{apcF0}, \ref{apcF1}
and \ref{apcF} through \ref{t1-xxx}
are presented in Section \ref{s2}.
The proof of Theorem \ref{t-bf} is similar to that of
\cite[Theorem 1.3(ii)]{dgyy0}.
We write $f-B_tf$ as a convolution operator, then control $f-B_tf$ by
some maximal functions via calculating pointwise estimates of the related
operator kernel and finally apply the Fefferman-Stein vector-valued maximal
inequality (see, for example, \cite{fs71}). The Calder\'on
reproducing formula on $\rn$ also plays a key role in this proof.
By means of Theorem \ref{t-bf}, together with some known characterizations
of $ F^\az_{p,q}(\rn)$ via Lusin-area functions involving differences,
we then prove Theorems \ref{apcF2} through \ref{apcF1}.
Using these characterizations in the limiting case $q=\fz$,
in Theorems \ref{t-bf} through \ref{apcF0}, of $F^\az_{p,q}(\rn)$,
we obtain the pointwise characterizations of $F^\az_{p,\fz}(\rn)$ in
Theorems \ref{apcF} through \ref{t1-xxx}. This method is totally different
from the method used in the proofs of \cite[Theorems 1.1 through 1.4]{dgyy},
which strongly depends on the
behaviors of the Laplace operator on $\rn$ and is available only
for Besov and Triebel-Lizorkin spaces with even smoothness orders
and hence is not suitable for Theorems \ref{apcF} through \ref{t1-xxx}
in this article, since Theorems \ref{apcF} through \ref{t1-xxx}
concern Triebel-Lizorkin spaces with fractional smoothness orders.

Finally, Section \ref{s3} is devoted to some corresponding results of
Theorems \ref{t-bf}, \ref{apcF2}, \ref{apcF0}, \ref{apcF1} and \ref{apcF} through \ref{t1-xxx} for Triebel-Lizorkin spaces  with smoothness order bigger than $2$. We also show some items in Theorems \ref{apcF} through
\ref{t1-xxx} are still equivalent on spaces of homogeneous type in the sense of Coifman and Weiss.

To end this section, we make some conventions on notation.
We use the \emph{symbol} $A \ls B$ to denote that
there exists a positive constant $C$ such that $A \le C \,B$.
The symbol $A \sim B$ is used as an abbreviation of
$A \ls B \ls A$. Here and hereafter, the \emph{symbol}  $C $ denotes a positive constant
which is independent of the main parameters, but may depend on the fixed
parameters $n,\,\az,\,p,\,q,\,\lambda$ and also probably auxiliary functions,
unless otherwise stated; its value  may vary from line to line. For any $p\in[1,\fz)$,
let $p'$ denotes its \emph{conjugate index}, namely, $1/p+1/p'=1$.

\section{Proofs of Theorems \ref{t-bf}, \ref{apcF2}, \ref{apcF0}, \ref{apcF1} and \ref{apcF} through \ref{t1-xxx}}\label{s2}
\hskip\parindent
First, we give the proof of Theorem \ref{t-bf}. To this end, we need some technical lemmas.
For all $t\in(0,\fz)$ and $x\in\rn$, let
$I(x):=\f 1 {|B(0,1)|} \chi_{B(0,1)} (x)$ and  $I_t(x):=t^{-n} I(x/t)$.
Then
 $$B_t f(x) =( f * I_{t})(x),\quad x\in\rn,\ \ t\in(0,\fz),$$
and hence
\begin{equation*}
(B_t f)^{\we} (\xi) = \wh{I} (t\xi) \wh{f} (\xi),\quad    \  \xi\in\rn.
\end{equation*}
It is easy to check that
\begin{equation*}
 \wh{I}(x) =\ga_n\int_0^1 \cos (u|x|) (1-u^2) ^{\f {n-1}2} \, du, \    \  x\in\rn,
 \end{equation*}
 with
 $
 \ga_n := [ \int_0^1(1-u^2) ^{\f {n-1}2} \, du]^{-1}
 $
(see also \cite[p.\,430, Section 6.19]{s93}).

For all $\lz,\,q\in(1,\fz),\,\bz\in(0,\fz)$, non-negative measurable
functions $F:\rn\times (0,\fz)\to\mathbb{C}$ and $x\in\rn$, define
\begin{eqnarray*}
\mathcal{G}(F)(x):= \lf\{\int_0^1 |F(x,t)|^q\,\frac{dt}t\r\}^{\frac1q},
\end{eqnarray*}
\begin{eqnarray*}\mathcal{S}_\bz(F)(x):=\lf\{\int_0^1 \dashint_{B(x,\bz t)} |F(y,t)|^q \,dy\, \frac{dt}t\r\}^{\frac1q}
\end{eqnarray*}
and
\begin{eqnarray*}
\mathcal{G}_\lambda^*(F)(x):=\lf\{\int_0^1 \int_\rn |F(y,t)|^q
  \lf(\frac t {t+|x-y|}\r)^{\lz n} \,dy\,
  \frac{dt}{t^{n+1}}\r\}^{\frac1q}.
\end{eqnarray*}
We write $\mathcal{S}(F):=\mathcal{S}_1(F)$.

We have the following technical lemma.
\begin{lemma}\label{lem-2-3}
Let $\lz,\,p,\,q,\,\bz\in(1,\fz)$.
Then there exists a positive constant $C$ such that, for all measurable functions
$F$ on $\rn\times (0,\fz)$,
\begin{enumerate}
  \item[\emph{(i)}] for all $x\in\rn$, $\mathcal{S}(F)(x)\le C\mathcal{G}_\lz^\ast(F)(x)$;
  \item[\emph{(ii)}]  $$\lf\|\mathcal{S}_\bz(F)\r\|_{L^p(\rn)}\le C\bz^{n(\frac1{\min\{p,q\}}-\frac1q)}\|S(F)\|_{L^p(\rn)},$$ where
  $C$ is independent of $\bz$ and $F$;
  \item[\emph{(iii)}] for $p\in[q,\fz)$, $\|\mathcal{G}_\lz^\ast(F)\|_{L^p(\rn)}\le C\|\mathcal{G}(F)\|_{L^p(\rn)}$.
\end{enumerate}
\end{lemma}
 \begin{proof}The proof of Lemma \ref{lem-2-3}(i) is obvious. Similar to the proofs of \cite[Theorem 4.4 and (4.3)]{Torchinsky}, we can prove that Lemma \ref{lem-2-3}(ii) holds true for $p\in[q,\fz)$.
Now, we give the proof of Lemma \ref{lem-2-3}(ii) for $p\in(1,q)$. To this end, For all $\mu\in(0,\fz)$ and measurable functions
$F$ on $\rn\times (0,\fz)$,
let $E_{\mu}:=\{x\in\rn:\ S(F)(x)>\mu\beta^{n/q}\}$  and
$$U_\mu:=\{x\in\rn:\ M(\chi_{E_{\mu}})(x)>(4\beta)^{-n}\},$$
where $M$ denotes the Hardy-Littlewood maximal function. Then, by the
week type $(1,1)$ boundedness of $M$, we see that,
for all $\mu\in(0,\fz)$,
\begin{eqnarray}\label{az1}
|U_\mu|
\ls (4\beta)^{n}\|\chi_{E_{\mu}}\|_{L^1(\rn)}
\sim \beta^{n}|E_{\mu}|.
\end{eqnarray}
Let
$$\rho(y):=\inf\lf\{|y-z|:\ z\in U_\mu^\com\r\},$$
where $U_\mu^\com:=\rn\setminus U_\mu$.
Then, by the Fubini theorem, it holds true that
\begin{eqnarray}\label{bz3}
&&\int_{ U_\mu^\com}\lf[\mathcal{S}_\beta(F)(x)\r]^q\,dx\\
&&\hs=\int_{U_\mu^\com}\int_0^1 \int_{\{y\in\rn:\ |y-x|<\bz t\}}
|F(y,t)|^q(\bz t)^{-n}\,\frac{dy\,dt}t\,dx\noz\\
&&\hs=\int_0^1
\int_{\{y\in\rn:\ \rho(y)<\bz t\}}
|F(y,t)|^q\lf|U_\mu^\com\cap B(y,\beta t)\r|(\bz t)^{-n}\,\frac{dy\,dt}t.\noz
\end{eqnarray}
If $ U_\mu^\com\cap B(y,\bz t)\neq\emptyset$,
then there exists $x_0\in U_\mu^\com\cap B(y,\bz t)$ and,
by the definition of $ U_\mu$ and $\bz\in(1,\fz)$, we see that
$$\frac{| E_{\mu}\cap B(y,t)|}{|B(y,t)|}
\le\frac{\bz^n}{|B(y,\bz t)|}\int_{B(y,\bz t)}\chi_{E_{\mu}}(x)\,dx
\le \bz^nM(\chi_{E_{\mu}})(x_0)\le 4^{-n},$$
which further implies that
\begin{eqnarray}\label{bz4}
\lf|U_\mu^\com\cap B(y,\bz t)\r|
&&\ls\bz^{n}\frac{|B(y,t)|}{|E_{\mu}^\com\cap B(y,t)|}
\lf| E_{\mu}^\com\cap B(y,t)\r|
\ls\bz^{n}\lf| E_{\mu}^\com\cap B(y,t)\r|.
\end{eqnarray}
If $ U_\mu^\com\cap B(y,\bz t)=\emptyset$, then \eqref{bz4} still holds true.
Thus, from \eqref{bz3} and \eqref{bz4},
it follows that
\begin{eqnarray*}
&&\int_{ U_\mu^\com}\lf[\mathcal{S}_\bz(F)(x)\r]^q\,dx\\
&&\hs\ls\int_0^1\int_\rn
|F(y,t)|^q\lf|E_{\mu}^\com\cap B(y,t)\r|(\bz t)^{-n}\bz^{n}\,\frac{dy\,dt}t\noz\\
&&\hs\ls\int_{ E_{\mu}^\com}\int_0^1 \int_{\{y\in\rn:\ |y-x|< t\}}
|F(y,t)|^q\,\frac{dy\,dt}{t^{n+1}}\,dx
\sim\int_{E_{\mu}^\com}\lf[\mathcal{S}(F)(x)\r]^q\,dx.\noz
\end{eqnarray*}
By this and \eqref{az1}, for all $\ell\in\mathbb{Z}$, we have
\begin{eqnarray}\label{five55}
&&\lf|\lf\{x\in\rn:\ \mathcal{S}_\bz(F)(x)>2^\ell\r\}\r|\\
&&\quad\leq\lf|U_{2^\ell}\r|+\lf|U_{2^\ell}^\com\cap\lf\{x\in\rn:\ \mathcal{S}_\bz(F)(x)>2^\ell\r\}\r|\noz\\
&&\quad\ls\beta^{n}\lf|E_{2^\ell}\r|+2^{-q\ell}\int_{E_{2^\ell}^\com}
[\mathcal{S}(F)(x)]^q\,dx\noz\\
&&\quad\sim\beta^{n}\lf|E_{2^\ell}\r|
+2^{-q\ell}\int_0^{2^\ell \beta^{\frac nq}}\nu^{q-1}\lf|\lf\{x\in\rn:\ S(F)(x)>\nu\r\}\r|\,d\nu\noz\\
&&\quad\ls\beta^{n}\lf|E_{2^\ell}\r|+2^{-q\ell}
\sum_{m=-\fz}^{m_\ell}\int_{2^{m-1}}^{2^m}\nu^{q-1}\lf|\lf\{x\in\rn:\ \mathcal{S}(F)(x)>\nu\r\}\r|\,d\nu\noz\\
&&\quad\ls\beta^{n}\lf|E_{2^\ell}\r|+2^{-q\ell}
\sum_{m=-\fz}^{m_\ell}2^{q(m-1)}
\lf|\lf\{x\in\rn:\ \mathcal{S}(F)(x)>2^{m-1}\r\}\r|,\noz
\end{eqnarray}
where $m_\ell:=\ell+\lfloor\frac nq\log_2\beta\rfloor+1$ and
$\lfloor s\rfloor$ denotes the biggest integer which does not exceed the real number $s$.

Therefore, when $p\in(1,q)$, by \eqref{five55}
and the definition of $E_{\mu}$, we know that
\begin{eqnarray*}
&&\lf\|\mathcal{S}_\bz(F)\r\|^p_{L^p(\rn)}\\&&\quad\sim\sum_{\ell\in\mathbb{Z}}2^{\ell p}
\lf|\lf\{x\in\rn:\ \mathcal{S}_\bz(F)(x)>2^{\ell}\r\}\r|\noz\\
&&\quad\ls\sum_{\ell\in\mathbb{Z}}2^{\ell p}
\beta^n\lf|\lf\{x\in\rn:\ \mathcal{S}(F)(x)>2^{\ell}\beta^{\frac nq}\r\}\r|\noz\\
&&\quad\quad+\sum_{\ell\in\mathbb{Z}}2^{\ell (p-q)}
\sum_{m=-\fz}^{m_\ell}2^{q(m-1)}\lf|\lf\{x\in\rn:\ \mathcal{S}(F)(x)>2^{m-1}\r\}\r|\noz\\
&&\quad\ls\beta^{(1-\frac pq)n}\lf\|\mathcal{S}(F)\r\|^p_{L^p(\rn)}\noz\\&&\quad\quad+
\sum_{\ell\in\mathbb{Z}}2^{(1-\gamma)p\ell}\beta^{\frac nq(q-p\gamma)}
\sum_{m=-\fz}^{m_\ell}2^{p\gamma(m-1)}\lf|\lf\{x\in\rn:\ \mathcal{S}(F)(x)>2^{m-1}\r\}\r|\noz\\
&&\quad\sim\beta^{(1-\frac pq)n}\lf\|\mathcal{S}(F)\r\|^p_{L^p(\rn)}\noz\\&&\quad\quad+
\sum_{m\in\mathbb{Z}}\beta^{\frac nq(q-p\gamma)}2^{p\gamma(m-1)}
\sum_{\ell=\ell_m}^{\fz}2^{(1-\gamma)p\ell}\lf|\lf\{x\in\rn:\ \mathcal{S}(F)(x)>2^{m-1}\r\}\r|\noz\\
&&\quad\ls\beta^{(1-\frac pq)n}\lf\|\mathcal{S}(F)\r\|^p_{L^p(\rn)}+\beta^{(1-\frac pq)n}
\sum_{m\in\mathbb{Z}}2^{p(m-1)}\lf|\lf\{x\in\rn:\ \mathcal{S}(F)(x)>2^{m-1}\r\}\r|\noz\\
&&\quad\sim\beta^{(1-\frac pq)n}\lf\|\mathcal{S}(F)\r\|^p_{L^p(\rn)},\noz
\end{eqnarray*}
where $\gamma\in(1,q/p)$ and $\ell_m:=m-\lfloor\frac nq\log_2\beta\rfloor-1$, which
finishes the proof of Lemma \ref{lem-2-3}(ii) for $p\in(1,q)$.

Now, we show Lemma \ref{lem-2-3}(iii).  By the Fubini theorem, we see that, for any non-negative measurable function
$h$,
 \begin{eqnarray*}
&&\int_{\rn}[\mathcal{G}_{\lz}^\ast(F)(x)]^qh(x)\,dx\\
&&\quad=\int_{\rn}\int_0^1 \int_{\mathbb  R^n}\lf(\frac{t}{t+|x-y|}\r)^{\lambda
 n}\,|F(y,t)|^q\frac{dy\,dt}{t^{n+1}}h(x)\,dx\\
&&\quad=\int_0^1\int_{\rn}|F(y,t)|^q\int_{\mathbb  R^n}\lf(\frac{t}{t+|x-y|}\r)^{\lambda
 n}\frac1{t^n}h(x)\,dx\,dy\,\frac{dt}{t}\\
&&\quad\le\int_{\rn}\int_0^1|F(y,t)|^q\lf[\sup_{t\in(0,\fz)}\int_{\mathbb  R^n}\frac{t^{(\lambda-1)n}}{\lf(t+|x-y|\r)^{\lambda
 n}}h(x)\,dx\r]\,\frac{dt}{t}\,dy\\
&&\quad\ls\int_{\rn}[\mathcal{G}(F)(y)]^q Mh(y)\,dy.
\end{eqnarray*}

Therefore, by $p\geq q$ and the boundedness of $M$ on $L^{(p/q)'}(\rn)$,
we find that
\begin{eqnarray*}
\lf\|\mathcal{G}_{\lz}^\ast(F)\r\|_{L^p(\rn)}^q
&&=\lf\|\lf[\mathcal{G}_{\lz}^\ast(F)\r]^q\r\|_{L^{p/q}(\rn)}\\
&&=\sup_{\|h\|_{L^{(p/q)'}(\rn)}\leq1}
\int_{\rn}[\mathcal{G}_{\lz}^\ast(F)(x)]^qh(x)\,dx\\
&&\ls\sup_{\|h\|_{L^{(p/q)'}(\rn)}\leq1}\int_{\rn}[\mathcal{G}(F)(x)]^q Mh(x)\,dx\\
&&\ls\lf\|\mathcal{G}(F)\r\|_{L^p(\rn)}^q
\sup_{\|h\|_{L^{(p/q)'}(\rn)}\leq1}\|Mh\|_{L^{(p/q)'}(\rn)}
\ls\lf\|\mathcal{G}(F)\r\|_{L^p(\rn)}^q,
\end{eqnarray*}
which implies Lemma \ref{lem-2-3}(iii) holds true and hence finishes the proof of Lemma \ref{lem-2-3}.
\end{proof}
The following two lemmas come from \cite[Lemmas 2.1 and 2.2]{dgyy0}, respectively.
\begin{lemma}\label{lem-2-1}
For all $x\in\rn$,
\begin{equation*}
\wh{I}(x) =1-A(|x|),\end{equation*}
where
\begin{equation*}
A(s) := 2\ga_n \int_0^1 (1-u^2)^{\f
{n-1}2} \lf(\sin \f {us}2\r)^2 \, du,\   \   \  s\in\mathbb{R}.\end{equation*}
Furthermore, $s^{-2} A(s)$ is a   smooth function on $\mathbb{R}$ satisfying that
there exist positive constants $c_1$ and $c_2$ such that
\begin{equation}\label{2-5}
0<c_1 \leq \f {A (s)}{s^2} \leq c_2,\    \    \ s\in(0,4]
\end{equation}
and
$$\sup_{s\in \mathbb{R}}
\lf| \lf(\f {d} {ds}\r)^i \lf(\f {A(s)}{s^2}\r)\r| <\infty,\
\   \  i\in\nn.$$
\end{lemma}

\begin{lemma}\label{lem-2-2}  Let  $\{T_t\}_{t\in(0,\fz)}$ be a family of multiplier
operators given  by setting, for all $f\in L^2(\rn)$,
 $$ ( T_t f)^{\we} (\xi) := m(t\xi) \wh{f}(\xi),\quad     \     \xi\in \rn,\   \   t\in(0,\fz)$$
 for some $ m\in L^\infty(\rn)$. If
 $$ \|\nabla^{n+1}  m \|_{L^1(\rn)}+\|m\|_{L^1(\rn)}  \leq C_1<\infty, $$
 then there exists a positive constant $C$ such that, for all $f\in L^2(\rn)$ and $x\in\rn$,
 $$ \sup_{t\in(0,\fz)} | T_t f(x)|\le  C C_1Mf(x).$$
\end{lemma}

The proof of Theorem \ref{t-bf} is similar to that of \cite[Theorem 3.1]{dgyy0},
which is a `discrete' version of Theorem \ref{t-bf}. Observing that only a sketch
of the proof of \cite[Theorem 3.1]{dgyy0} was given, for the
sake of completeness, we give the proof of Theorem \ref{t-bf} here.

\begin{proof}[Proof of Theorem \ref{t-bf}]
Let $\varphi$ and $\Phi$ satisfy \eqref{con} and \eqref{con1}, respectively.
Then there exist Schwartz functions $\psi$ and $\Psi$ satisfying
\eqref{con} and \eqref{con1}, respectively, such that
$$\wh\Phi(\xi)\wh\Psi(\xi)+\int_0^1 \wh\vz(t\xi)\wh\psi(t\xi)\,\frac{dt}t=1,\quad \xi\in\rn;$$
see, for example, \cite{a77,a75}.

Let $f\in F^\az_{p,\,q}(\rn)$.
Then it is well known that $f\in L^p(\rn)$ (see \cite[Theorem/2.5.11]{t83}).
Moreover, the equality
\begin{equation}\label{crf}
f=\Phi\ast \Psi \ast f+\int_0^{1}\vz_t\ast\psi_t\ast f\,\frac{dt}t
\end{equation}
holds true both in $L^p(\rn)$ and $\cs'(\rn)$, due to the
Calder\'on reproducing formula (see, for example, \cite{a77}).
Now we show $|||f|||_{\F(\rn)}\ls \|f\|_{\F(\rn)}$, and it suffices to prove that
 \begin{equation}\label{3-1}
\left\|\lf[\int_0^1 t^{-\az q}
|f-B_tf|^q\,\frac{dt}t\r]^{1/q}\right\|_{L^p(\rn)}\ls \|f\|_{\F(\rn)},
 \end{equation}
 since $\|f\|_{L^p(\rn)}\ls \|f\|_{\F(\rn)}.$

Indeed, by \eqref{crf}, for all $s,\,t\in(0,1)$ and $\xi\in\rn$,
 \begin{eqnarray*}
 (f- B_sf )^{\we} (\xi)&&= \wh\Phi(\xi)A(s|\xi|) \wh{f_1}(\xi)+\int_0^1 \wh{\vi} (t \xi) A ( s |\xi|)  \wh{f_t}(\xi)\,\frac{dt}t\\
&& =:  (T_{s, 1} f_1 )^{\we} (\xi)+\int_0^1  (T_{s, t} f_t )^{\we} (\xi)\,\frac{dt}t,
 \end{eqnarray*}
where  $T_{s, t}$ is given  by
 \begin{equation}\label{3-2}
 (T_{s, t} f_t )^{\we} (\xi):=\wh{\vi} (t \xi) A ( s |\xi|)  \wh{f_t}(\xi), \quad t\in(0,1),\ \ \xi\in\rn,
 \end{equation}
and
\begin{equation*}
 (T_{s, 1} f_1 )^{\we} (\xi):=\wh\Phi (\xi) A ( s |\xi|)  \wh{f_1}(\xi), \quad \xi\in\rn,
 \end{equation*}
with $\wh{f_t}:=\wh{\psi}(t\cdot)\wh{f}$ and $\wh{f_1}:=\wh{\Psi}(\cdot)\wh{f}$.
Therefore,
\begin{equation} \label{3-3}
f- B_s f = T_{s, 1} f_1 +\int_0^1  T_{s, t} f_t \,\frac{dt}t.
\end{equation}

For the integral part in \eqref{3-3}, we split $\int_0^1$ into two parts $\int_0^s$ and $\int_s^1$.
It is relatively easier
to deal with the first part. Indeed, for $t\in(0,s]$, by \eqref{3-2}, we find  that, for all $x\in\rn$,
\begin{eqnarray*}
|T_{s,t} f_t(x)|& = |( I- B_s ) ( f\ast\psi_t\ast \vi_t) (x)|
\ls M (f\ast\psi_t\ast \vi_t)(x).
 \end{eqnarray*}
From this, $\az\in(0,2)$ and the H\"{o}lder inequality, we deduce that
\begin{eqnarray}\label{3-4}
\int_0^1 s^{-\az q}
\lf|\int_0^s  T_{s, t} f_t \,\frac{dt}t\r|^q\,\frac{ds}s
&&\ls  \int_0^1 s^{-\az q}
\lf[\int_0^s  M (f\ast\psi_t\ast \vi_t) \,\frac{dt}t\r]^q\,\frac{ds}s\\
&&\ls \int_0^1 s^{-\frac{\az q}2}
\int_0^s \lf[M (f\ast\psi_t\ast \vi_t)\r]^q t^{-\frac{\az q}2} \,\frac{dt}t\,\frac{ds}s\noz\\
&&\ls \int_0^1 t^{-\az q} \lf[M (f\ast\psi_t\ast \vi_t)\r]^q \,\frac{dt}t\noz.
\end{eqnarray}

Now we  estimate the integral $\int_s^1$. For all $t\in(0,1)$, $s\in(0,t)$
and $\xi\in \rn$, write
$$(T_{s, t} f_t )^{\we} (\xi)=\wh{\vi} (t \xi) A ( s |\xi|)  \wh{f_t}(\xi)=:m_{s,t}(\xi) \wh{f_t}(\xi),$$
  where
\begin{eqnarray*}
m_{s,t} (\xi) :=
\wh{\vi} (t \xi)\f {A ( s |\xi|)} { ( s |\xi|)^2  } ( s |\xi|)^2,\qquad \xi\in\rn.
\end{eqnarray*}
Write $\wz m_{s,t}(\xi):= m_{s,t} (t^{-1}\xi)$. By Lemma \ref{lem-2-1}, we see that,
for all $t\in(s,1)$ and $\xi\in \rn$,
\begin{equation*}
 |\partial^\bz \wz m_{s,t} (\xi)| \ls  \lf(\f st\r)^2 \chi_{\overline{B(0,2)}\setminus B(0,1/2)}(\xi), \qquad    \bz\in\zz_+^n,
\end{equation*}
and thus
\begin{eqnarray*}
\|\wz m_{s,t}\|_{L^1(\rn)}
+\|\nabla^{n+1}\wz m_{s,t}\|_{L^1(\rn)}\ls  \lf(\f st\r)^2,
\end{eqnarray*}
which, together with Lemma \ref{lem-2-2}, further implies that
$$|T_{s,t} f_t(x)| \ls \lf(\f st\r)^2 M f_t(x),\qquad x\in\rn.$$
By this, for $\a\in (0, 2)$, taking $\bz:=1-\az/2>0$, together with
the H\"{o}lder inequality, we conclude that
\begin{eqnarray}\label{3-5}
\int_0^1 s^{-\az q}
\lf|\int_s^1  T_{s, t} f_t \,\frac{dt}t\r|^q\,\frac{ds}s
&&\ls  \int_0^1 s^{(2-\az) q}
\lf[\int_s^1  t^{-2}M (f_t) \,\frac{dt}t\r]^q\,\frac{ds}s\\
&&\ls \int_0^1 s^{(2-\az-\bz) q}
\int_s^1 \lf[M (f_t)\r]^q t^{(\bz-2) q} \,\frac{dt}t\,\frac{ds}s\noz\\
&&\ls \int_0^1 t^{(\bz-2) q} \lf[M (f_t)\r]^q \lf(\int_0^t s^{(2-\az-\bz) q} \,\frac{ds}s\r)\,\frac{dt}t
\noz\\
&&\sim \int_0^1 t^{-\az q} \lf[M (f_t)\r]^q \,\frac{dt}t\noz.
\end{eqnarray}

For the part  $T_{s, 1} f_1$ in \eqref{3-3}, we make use of the idea used in the above estimate for  $\int_s^1$, and find that
\begin{equation}\label{3-5x}
|T_{s,1} f_1(x)| \ls s^2 M (f_1)(x),\qquad x\in\rn.
\end{equation}

Combining \eqref{3-4}, \eqref{3-5} and \eqref{3-5x}  with \eqref{3-3}, using the Fefferman-Stein vector-valued maximal inequality (see \cite{fs71}), the
independence of
$\F(\rn)$ on the pair $(\vz,\Phi)$ (see Remark \ref{sr1}(i)) and Remark \ref{sr1}(ii),
we see that
\begin{eqnarray*}
&& \left\|\lf[\int_0^1 s^{-\az q}
|f-B_sf|^q\,\frac{ds}s\r]^{1/q}\right\|_{L^p(\rn)}\\
&&\quad \ls  \left\|\lf[\int_0^1 s^{(2-\az) q}
\,\frac{ds}s\r]^{1/q}M (f\ast \Psi)\right\|_{L^p(\rn)} +\left\|\lf\{\int_0^1 t^{-\az q}
\lf[M (f\ast \psi_t\ast \vz_t)\r]^q\,\frac{dt}t\r\}^{1/q}\right\|_{L^p(\rn)}\\
&&\quad\quad+\left\|\lf\{\int_0^1 t^{-\az q}
\lf[M (f\ast \psi_t)\r]^q\,\frac{dt}t\r\}^{1/q}\right\|_{L^p(\rn)}
\\
&&\quad\ls \|f\ast \Psi\|_{L^p(\rn)}+\left\|\lf[\int_0^1 t^{-\az q}
\lf|f\ast \psi_t\r|^q\,\frac{dt}t\r]^{1/q}\right\|_{L^p(\rn)}\\
&&\quad\quad+\left\|\lf[\int_0^1 t^{-\az q}
\lf|f\ast \psi_t\ast \vz_t\r|^q\,\frac{dt}t\r]^{1/q}\right\|_{L^p(\rn)}\\
&&\quad\sim\|f\|_{\F(\rn)},
\end{eqnarray*}
which proves \eqref{3-1}.

To show the inverse direction, we only need to prove
\begin{equation}\label{3-6}
\|f\|_{\F(\rn)}\ls \|f\|_{L^p(\rn)}+\left\|\lf[\int_0^1 s^{-\az q}
|f-B_sf|^q\,\frac{ds}s\r]^{1/q}\right\|_{L^p(\rn)}
\end{equation}
whenever $f\in L^p(\rn)$ and the right-hand side of \eqref{3-6}
is finite.
For this purpose, we first claim that
\begin{equation}\label{3-7}
|f\ast \vi_t(x)| \ls M ( f -B_t f)(x),\qquad t\in(0,1),\ \ x\in\rn.
\end{equation}
Indeed, we find that, for all $t\in(0,1)$ and $\xi\in\rn$,
\begin{align}\label{se9}
 ( f\ast \vi_t )^{\we}  (\xi) = \f {\wh{\vi} (t\xi) } { A ( t|\xi|)} ( f -B_t f)^{\we} (\xi) =:\eta(t\xi) ( f -B_t f)^{\we} (\xi),
\end{align}
where $\eta(\xi):=\f { \wh{\vi} (\xi) } { A(|\xi|)}$ for all $\xi\in\rn$, which is well defined due to \eqref{2-5}.
By Lemma \ref{lem-2-1}, we see that
$\eta\in C_c^\infty(\rn)$ and $\supp\eta \subset \{\xi\in\rn:  \     \  \f12 \leq |\xi|\leq 2\}$.  The claim \eqref{3-7} then follows from Lemma \ref{lem-2-2}.

On the other hand, it is easy to see that
$\|\Phi\ast f\|_{L^p(\rn)}\ls \|f\|_{L^p(\rn)}$.
From this, Remark \ref{sr1}(ii), \eqref{3-7} and the Fefferman-Stein vector-valued maximal inequality (see \cite{fs71}),
we deduce that
\begin{eqnarray*}
\|f\|_{\F(\rn)}&&\sim\|\Phi\ast f\|_{L^p(\rn)}+\lf\|\lf[\int_0^1 t^{-\az q}
\dashint_{B(\cdot,\,t)}\lf|\vz_t\ast
f(y)\r|^q\,dy\,\frac{dt}t\r]^{1/q}\r\|_{L^p(\rn)}\\
&&\ls  \|f\|_{L^p(\rn)}+ \left\|\lf\{\int_0^1 s^{-\az q}
\lf[M(f-B_sf)\r]^q\,\frac{ds}s\r\}^{1/q}\right\|_{L^p(\rn)}\\
&& \ls \|f\|_{L^p(\rn)}+\left\|\lf[\int_0^1 s^{-\az q}
|f-B_sf|^q\,\frac{ds}s\r]^{1/q}\right\|_{L^p(\rn)}\\
&&\sim\||f|\|_{\F(\rn)}.
\end{eqnarray*}
This finishes the proof of Theorem \ref{t-bf}.
\end{proof}

Now we prove Theorem \ref{apcF2}.
\begin{proof}[Proof of Theorem \ref{apcF2}]
Let all notation be the same as in the proof of Theorem \ref{t-bf}.
We first prove (i)$\Longrightarrow$(ii). Let $f\in \F(\rn)$.
By the Fefferman-Stein
vector-valued maximal inequality (see \cite{fs71}) and Theorem \ref{t-bf}, we see that, for all $r\in[1,q)$,
\begin{eqnarray*}
 &&\left\|\lf\{\int_0^1 s^{-\az q}
\lf[\dashint_{B(\cdot,s)}\lf|f-B_sf\r|^r\r]^{\frac qr}
\frac{ds}s\r\}^{1/q}\right\|_{L^p(\rn)}\\
&&\hs\ls \left\|\lf\{\int_0^1 s^{-\az q}
\lf[M\lf(\lf|f-B_sf\r|^r\r)\r]^{\frac qr}
\frac{ds}s\r\}^{1/q}\right\|_{L^p(\rn)}\\
&&\hs\ls\left\|\lf\{\int_0^1 s^{-\az q}
\lf|f-B_sf\r|^q
\frac{ds}s\r\}^{1/q}\right\|_{L^p(\rn)}
\ls\|f\|_{\F(\rn)},
 \end{eqnarray*}
which finishes
the proof of (i)$\Longrightarrow$(ii).

Conversely, we show (ii)$\Longrightarrow$(i). Since $\eta$ in \eqref{se9} is a Schwartz function, by \eqref{se9},
we observe that, for all $t\in(0,1)$ and $x\in\rn$,
\begin{eqnarray*}
\dashint_{B(x,t)}\lf|\vi_t\ast f (y)\r|\,dy&&=
\dashint_{B(x,t)}\lf|(\eta(t\cdot))^\vee\ast (f -B_t f)(y)\r|\,dy\\
&&\ls\int_{\rn}\lf|(\eta(t\cdot))^\vee(z)\r|\dashint_{B(x,t)}\lf|(f -B_t f)(y-z)\r|\,dy\,dz\\
&&\ls M\lf(\dashint_{B(\cdot,t)}\lf|(f -B_t f)(y)\r|\,dy\r)(x).
\end{eqnarray*}
From this, by Remark \ref{sr1}(ii), the Fefferman-Stein
vector-valued maximal inequality (see \cite{fs71}) and the H\"{o}lder inequality, we find that, for all $r\in[1,q)$,
\begin{eqnarray*}
\|f\|_{F^\az_{p,\,q}(\rn)}
 &&\sim\|\Phi\ast f\|_{L^p(\rn)}+\lf\|\lf\{\int_0^1 t^{-\az q}
\lf[\dashint_{B(\cdot,\,t)}\lf|\vz_t\ast
f(y)\r|\,dy\r]^q\,\frac{dt}t\r\}^{1/q}\r\|_{L^p(\rn)}\\
 &&\ls\|\Phi\ast f\|_{L^p(\rn)}\\
 &&\quad+\lf\|\lf\{\int_0^1 t^{-\az q}
\lf[M \lf(\dashint_{B(\cdot,t)} \lf|(f -B_t f)(y)\r|\,dy\r)\r]^q\,\frac{dt}t\r\}^{1/q}\r\|_{L^p(\rn)}\\
&&\ls\|\Phi\ast f\|_{L^p(\rn)}+\left\|\lf\{\int_0^1 t^{-\az q}\lf[\dashint_{B(\cdot,t)}
|(f-B_tf)(y)|\,dy\r]^q\,\frac{dt}t\r\}^{1/q}\right\|_{L^p(\rn)}\\
&&\ls\|\Phi\ast f\|_{L^p(\rn)}+\left\|\lf\{\int_0^1 t^{-\az q}\lf[\dashint_{B(\cdot,t)}
|(f-B_tf)(y)|^r\,dy\r]^{\frac qr}\,\frac{dt}t\r\}^{1/q}\right\|_{L^p(\rn)}.
\end{eqnarray*}
This finishes the proof of (ii)$\Longrightarrow$(i) and hence the proof of Theorem \ref{apcF2}.
\end{proof}

Now, we prove Theorem \ref{apcF0}.
\begin{proof}[Proof of Theorem \ref{apcF0}]
Let all notation be the same as in the proof of Theorem \ref{t-bf}.

We first prove (i). If $p\in[q,\fz)$ and $\al\in(0,2)$, then the desired conclusion follows from Theorem \ref{t-bf}, and
 (i) and (iii) of Lemma \ref{lem-2-3}. Now we assume that $p\in(1,q)$ and $\a\in (n(1/p-1/q),1)$.
 Let $f\in \F(\rn)$.
Notice that, for all $t\in(0,1)$ and $x\in\rn$, we have
\begin{eqnarray*}
\dashint_{B(x,t)}\lf|f(y)-B_tf(y)\r|^q\,dy
&&\ls\dashint_{B(x,t)}\dashint_{B(y,t)}\lf|f(y)-f(z)\r|^q\,dz\,dy\\
&&\ls\dashint_{B(x,2t)}\lf|f(y)-f(x)\r|^q\,dy.
\end{eqnarray*}
From this and Remark \ref{sr1}(ii), we deduce that
\begin{eqnarray*}
&&\lf\|\lf[\int_0^1 t^{-\az q}
\dashint_{B(\cdot,\,t)}\lf|f(y)-B_tf(y)\r|^q\,dy\frac{dt}t\r]^{1/q}\r\|_{L^p(\rn)}\\
&&\hs\ls\lf\|\lf[\int_0^1 t^{-\az q}
\dashint_{B(\cdot,\,t)}\lf|f(y)-f(\cdot)\r|^q\,dy\frac{dt}t\r]^{1/q}\r\|_{L^p(\rn)}\\
&&\hs\ls\|f\|_{\F(\rn)},
\end{eqnarray*}
which finishes
the proof of Theorem \ref{apcF0}(i) .

Now we show (ii). Notice that, if $\widetilde{\||f|\|}_{\F(\rn)}<\fz$, then
$\||f|\|_{\F(\rn)}^{(1)}<\fz$ due to the H\"older inequality. Then,
by Theorem \ref{apcF2} and the H\"older inequality, we have
\begin{eqnarray*}
\|f\|_{F^\az_{p,\,q}(\rn)}
&&\ls\|f\|_{L^p(\rn)}+\lf\|\lf[\int_0^1 t^{-\az q}
\lf[\dashint_{B(\cdot,t)}\lf| (f -B_t f)(y)\r|\,dy\r]^q\,\frac{dt}t\r]^{1/q}\r\|_{L^p(\rn)}\\
&&\ls\|f\|_{L^p(\rn)}+\left\|\lf[\int_0^1 t^{-\az q}\dashint_{B(\cdot,t)}
|(f-B_tf)(y)|^q\,dy\,\frac{dt}t\r]^{1/q}\right\|_{L^p(\rn)}\\
&&\sim\widetilde{\||f|\|}_{\F(\rn)}.
\end{eqnarray*}
This finishes the proof of Theorem \ref{apcF0}(ii) and hence the proof of Theorem \ref{apcF0}.
\end{proof}

Now we employ Theorems \ref{t-bf} and \ref{apcF0} to prove Theorem \ref{apcF1}.
\begin{proof}[Proof of Theorem \ref{apcF1}]
We first show (i). Let $f\in \F(\rn)$.
For the case when $\a\in (n(1/p-1/q),1)$ and $p\in(1,\,q)$,
 by Lemma \ref{lem-2-3}(ii) with

\begin{equation*}
F:=\mathcal{F}_\al(x,t):=\lf|\frac{B_tf(x)-f(x)}{t^\al}\r|,\quad (x,t)\in\rn\times (0,\fz),
\end{equation*} we see that, for all $\lambda\in(q/p,\fz)$ and $x\in\rn$,
 $$\lf\|\mathcal{S}_{\bz}(\mathcal{F}_\al)\r\|_{L^p(\rn)}\ls
 \bz^{n(\frac1p-\frac1q)}\|\mathcal{S}(\mathcal{F}_\al)\|_{L^p(\rn)},$$
 which, combined with
\begin{eqnarray*}
\lf[\mathcal{G}_{\lz}^\ast(\mathcal{F}_\al)\r]^q
&&=\int_{0}^{1}t^{-\alpha q}
 \int_{|x-y|<t}\lf(\frac{t}{t+|x-y|}\r)^{\lambda
 n}\lf|f(y)-B_tf(y)\r|^q\frac{dy\, dt}{t^{n+1}}\\
 &&\hs+\sum_{k=1}^{\infty}\int_{0}^{1}
 \int_{2^{k-1}t\le|x-y|<2^kt}\cdots\\
&&\le\sum_{k=0}^{\infty}2^{-kn(\lambda-1)}
\lf[\mathcal{S}_{2^k}(\mathcal{F}_\al)\r]^q\noz
\end{eqnarray*}
and $\lambda/q>1/p$, further implies that
\begin{eqnarray*}
\lf\|\mathcal{G}_{\lz}^\ast(\mathcal{F}_\al)\r\|_{L^p(\rn)}&&\ls\sum_{k=0}^{\infty}2^{-\frac kqn(\lambda-1)}
\lf\|\mathcal{S}_{2^k}(\mathcal{F}_\al)\r\|_{L^p(\rn)}\\
&&\ls\sum_{k=0}^{\infty}2^{-\frac kqn(\lambda-1)}2^{k(\frac1p-\frac 1q)n}\lf\|\mathcal{S}(\mathcal{F}_\al)\r\|_{L^p(\rn)}\\
&&\sim\lf\|\mathcal{S}(\mathcal{F}_\al)\r\|_{L^p(\rn)}.
\end{eqnarray*}
By this, we see that the desired conclusion follows from Theorem \ref{apcF0}.
For the case when $\a\in (0,2)$ and $p\in[q,\,\fz)$, the desired conclusion in Theorem \ref{apcF1}(i) follows from Lemma \ref{lem-2-3}(iii) and Theorem \ref{t-bf}.

 Now we show (ii). By Lemma \ref{lem-2-3}(i), we know that $\mathcal{S}(\mathcal{F}_\al)(x)\ls \mathcal{G}_{\lz}^\ast(\mathcal{F}_\al)(x)$
for all $x\in\rn$. Then for $f\in L^p(\rn)$ with
$||\mathcal{G}_{\lz}^\ast(\mathcal{F}_\al)||_{L^p(\rn)}<\fz$, by Theorem \ref{apcF0},
we see that $$\|f\|_{\F(\rn)}
\ls\|f\|_{L^p(\rn)}+||\mathcal{S}(\mathcal{F}_\al)||_{L^p(\rn)}
\ls\|f\|_{L^p(\rn)}+||\mathcal{G}_{\lz}^\ast(\mathcal{F}_\al)||_{L^p(\rn)}.$$
This finishes the proof of Theorem \ref{apcF1}(ii) and hence
 the proof of Theorem \ref{apcF1}.
\end{proof}

 Now we use Theorems \ref{t-bf} and \ref{apcF0} to prove Theorem \ref{apcF}.
\begin{proof}[Proof of Theorem \ref{apcF}] \emph{Step 1}.
Let $\az\in(0,2)$ and $p\in(1,\fz)$. We first show (i)$\Longrightarrow$(ii). Assume that $f\in F^\az_{p,\,\fz}(\rn)$. Then, by Theorem \ref{t-bf}, we have $f\in L^p(\rn)$ and
$$
\|f\|_{L^p(\rn)}+\lf\|\sup_{t\in(0,1)}t^{-\alpha}\lf|f-B_tf\r|\r\|_{L^p(\rn)}
\ls\lf\|f\r\|_{F^\az_{p,\,\fz}(\rn)}<\fz.$$
For any $x\in\rn$, let $g(x):=\sup_{t\in(0,1)}t^{-\alpha}|(f-B_tf)(x)|$. Clearly, we see that
$g\in L^p(\rn)$ and $$\lf|(f-B_tf)(x)\r|\leq t^{\al}g(x),\ \ \ x\in\rn.$$
Moreover, $\|f\|_{L^p(\rn)}+\|g\|_{L^p(\rn)}\ls\|f\|_{F^\az_{p,\,\fz}(\rn)}$. This proves (ii).

Next we show (ii)$\Longrightarrow$(i). Assume that $f\in L^p(\rn)$ and there exists a non-negative
$g\in L^p(\rn)$ such that  $\lf|(f-B_tf)(x)\r|\ls t^{\alpha}g(x)$
for all $t\in(0,1)$ and almost every $x\in\rn$. Thus,
$$\|f\|_{L^p(\rn)}+\lf\|\sup_{t\in(0,1)}t^{-\alpha}\lf|f-B_tf\r|\r\|_{L^p(\rn)}
\ls\|f\|_{L^p(\rn)}+\|g\|_{L^p(\rn)}<\fz,$$
which, together with Theorem \ref{t-bf}, implies that $f\in F^\az_{p,\,\fz}(\rn)$. This finishes the proof
of (i)$\Longleftrightarrow$(ii).

\emph{Step 2}. Let $\az\in(n/p,1)$ and $p\in(1,\fz)$. We now show (i)$\Longrightarrow$(iii).
Assume that $f\in F^\az_{p,\,\fz}(\rn)$. Then, by Theorem \ref{apcF0}, we have
$f\in L^p(\rn)$ and
$$
\|f\|_{L^p(\rn)}+\lf\|\sup_{t\in(0,1)}\sup_{x\in B(\cdot,t)}t^{-\alpha}\lf|(f-B_tf)(x)\r|\r\|_{L^p(\rn)}
\ls\lf\|f\r\|_{ F^\az_{p,\,\fz}(\rn)}<\fz.$$
For all $y\in\rn$, let $g(y):=\sup_{t\in(0,1)}\sup_{x\in B(y,t)}t^{-\alpha}|(f-B_tf)(x)|$. Clearly,
$g\in L^p(\rn)$ and, for all $t\in(0,1)$ and almost every $x\in\rn$ and $y\in B(x,t)$,
$$\lf|(f-B_tf)(x)\r|\leq t^{\al}g(y).$$

Finally, we show  (iii)$\Longrightarrow$(ii). Assume that $f\in L^p(\rn)$ and there exist a non-negative
$g\in L^p(\rn)$ and positive constants $C_1,\,C_2$ such that  $t^{-\alpha}\lf|(f-B_tf)(x)\r|\leq C_2g(y)$
for all $t\in(0,1)$ and almost every $x\in\rn$ and $y\in B(x,C_1t)$. Therefore,
\begin{eqnarray*}
\lf|(f-B_tf)(x)\r|\leq C_2t^\al\dashint_{B(x,C_1t)}g(y)\,dy\ls t^{\al}Mg(x).
\end{eqnarray*}
Noticing that $g\in L^p(\rn)$ implies $Mg\in L^p(\rn)$, we see that (ii) holds true and
hence the proof of Theorem \ref{apcF} is finished.
\end{proof}

\begin{remark}
By the above proof, we know that (iii)$\Longrightarrow$(ii) holds true for all $\al\in(0,2)$.
The condition $\al\in(n/p,1)$ is only used for the proof of (i)$\Longrightarrow$(iii).
\end{remark}
Now we prove Theorem \ref{t1-x}.
\begin{proof}[Proof of Theorem \ref{t1-x}] By the H\"older inequality, we immediately see
that (ii)$\Longrightarrow$(iii) for all $\al\in(0,2)$.

Next, we show (iii)$\Longrightarrow$(i) when $\al\in(0,2)$. Assume that $f\in L^p(\rn)$ and there exists a non-negative
$g\in L^p(\rn)$ such that, for all $t\in(0,1)$ and almost every $x\in\rn$,
\begin{equation*}
|f(x)-B_t f(x)|\ls t^\al \lf\{\dashint_{B(x,C_6t)} [g(y)]^q\,dy\r\}^{1/q}.
\end{equation*}
Then, for all $t\in(0,1)$ and almost every $x\in\rn$,
\begin{equation*}
|f(x)-B_t f(x)|
\ls t^\al\lf[M(g^q)(x)\r]^{1/q}.
\end{equation*}
Since $g\in L^p(\rn)$ and $q\in[1,p)$, it follows that $\lf[M(g^q)(x)\r]^{1/q}\in L^p(\rn)$, which, together with the equivalence between (i) and (ii) of Theorem \ref{apcF}, implies $f\in F^\az_{p,\,\fz}(\rn)$.
This proves (i).

Finally, we show (i)$\Longrightarrow$(ii) when $\al\in(n/p,1)$. Let $f\in F^\az_{p,\,\fz}(\rn)$.
Then, by the equivalence between (i) and (iii) of Theorem \ref{apcF}, we know that $f\in L^p(\rn)$ and there exist a non-negative
$g\in L^p(\rn)$ and positive constants $C_3,\,C_4$ such that, for all $t\in(0,1)$ and almost every $x\in\rn$ and $y\in B(x,C_3t)$,
$\lf|(f-B_tf)(x)\r|\leq C_4t^{\alpha}g(y)$. Therefore
$$\lf|(f-B_tf)(x)\r|\leq C_4 t^{\alpha}\inf_{y\in B(x,C_3t)}g(y)
\leq C_4 t^{\alpha}\dashint_{B(x,C_3t)}g(y)\,dy.$$
This prove (ii) and hence
 finishes the proof of  Theorem \ref{t1-x}.
\end{proof}

Now we prove Theorem \ref{t1-xx}.
\begin{proof}[Proof of Theorem \ref{t1-xx}]
By the H\"older inequality, we see that, for all $\al\in(0,2)$,
\begin{eqnarray*}
{\rm(ii)}\Longrightarrow{\rm(v)}\Longrightarrow{\rm(vii)}\Longrightarrow{\rm(vi)}
\end{eqnarray*}
 and
 \begin{eqnarray*}
{\rm(ii)}\Longrightarrow{\rm(iv)}\Longrightarrow{\rm(iii)}\Longrightarrow{\rm(vi)}.
\end{eqnarray*}
Therefore, to complete the proof, it suffices to show (i)$\Longrightarrow$(ii) and (vi)$\Longrightarrow$(i).

Now we prove (vi)$\Longrightarrow$(i) when $\al\in(0,2)$. Assume that $f$ satisfies (vi). Then there exist a non-negative $g\in L^p(\rn)$ and positive constants
$C$ and $\widetilde{C}$ such that, for all $t\in(0,1)$ and almost every $x\in\rn$,
\begin{eqnarray}\label{se11}
&&\dashint_{B(x,t)} \lf|f(y)-B_{Ct} f(y)\r|\,dy\\
&&\hs\ls t^\al \lf\{\dashint_{B(x,\widetilde{C}t)} [g(y)]^q\,dy\r\}^{1/q}\noz\\
&&\hs\ls t^\al\lf[M(g^q)(x)\r]^{1/q}.\noz
\end{eqnarray}
Notice that $g\in L^p(\rn)$ and $q\in[1,p)$ implies $\lf[M(g^q)(x)\r]^{1/q}\in L^p(\rn)$.
From this, combined with \eqref{se11} and Theorem \ref{apcF2}, we deduce that
\begin{eqnarray*}\|f\|_{ F^\az_{p,\,\fz}(\rn)}&&\ls \|f\|_{L^p(\rn)}+
\lf\|\sup_{t\in(0,1)}t^{-\al}\dashint_{B(\cdot,t)} \lf|f(y)-B_{Ct} f(y)\r|\,dy\r\|_{L^p(\rn)}\\
&&\ls\|f\|_{L^p(\rn)}+\lf\|\lf[M(g^q)\r]^{1/q}\r\|_{L^p(\rn)}<\fz,
\end{eqnarray*}
which implies $f\in F^\az_{p,\,\fz}(\rn)$ for all $\al\in(0,2)$. This proves (i).

Finally, We prove (i)$\Longrightarrow$(ii) when $\al\in(n/p,1)$. Let $f\in F^\az_{p,\,\fz}(\rn)$. Then, by the equivalence between (i) and (iii) in Theorem \ref{apcF}, we see that $f\in L^p(\rn)$ and there exist a non-negative
$g\in L^p(\rn)$ and a positive constants $C$ such that, for all $t\in(0,1)$, almost every $y\in\rn$ and $z\in B(y,Ct)$,
$\lf|f(y)-B_tf(y)\r|\ls t^{\alpha}g(z)$. Therefore, for almost every $x\in\rn$ and $y\in B(x,t)$
$$\lf|f(y)-B_tf(y)\r|\ls t^{\alpha}\inf_{z\in B(y,Ct)}g(z)
\ls t^{\alpha}\dashint_{B(y,Ct)}g(z)\,dz\ls t^{\alpha}\dashint_{B(x,(1+C)t)}g(z)\,dz.$$
Thus,
$$\sup_{y\in B(x,t)}\lf|f(y)-B_tf(y)\r|\ls t^{\alpha}\dashint_{B(x,(1+C)t)}g(z)\,dz.$$
This proves (ii) and hence finishes the proof of Theorem \ref{t1-xx}.
\end{proof}

Finally we prove Theorem \ref{t1-xxx}.
\begin{proof}[Proof of Theorem \ref{t1-xxx}]
By the H\"older inequality, it is easy to see that, for all $\al\in(0,2)$, (ii)$\Longrightarrow$(iii)$\Longrightarrow$(iv)
and (v)$\Longrightarrow$(vi)$\Longrightarrow$(vii).

Next we prove
(iv)$\Longrightarrow$(i) and (vii)$\Longrightarrow$(i) when $\al\in(0,2)$. If (iv) holds true, then, by Theorem
\ref{apcF2}, we see that (i) holds true; if (vii)
holds true, then Theorem \ref{t1-xx}(iii) holds true, which further implies (i).
On the other hand, from Theorem \ref{apcF2}, we deduce that (i) implies (iii) for $\al\in(0,2)$.

It remains to prove (i)$\Longrightarrow$(ii)  and (i)$\Longrightarrow$(v) when $\al\in(n/p,1)$. Indeed, if (i) holds true, then Theorem \ref{t1-xx}(ii) holds true, which further implies (ii) and (v). This finishes
the proof of Theorem \ref{t1-xxx}.
\end{proof}

\section{Further Remarks}\label{s3}
\hskip\parindent
In this section, we first generalize some items of Theorems \ref{t-bf}, \ref{apcF2}, \ref{apcF0}, \ref{apcF1} and \ref{apcF} through \ref{t1-xxx} to the higher order Triebel-Lizorkin spaces  with order bigger than 2. As a further application, we then prove that some items in Theorems \ref{apcF} through \ref{t1-xxx} are still
equivalent on spaces of homogeneous type,
which can be used to define the Triebel-Lizorkin spaces on spaces of homogeneous type with
the smoothness order $\al\in(0,2)$.

\subsection{Higher Order Triebel-Lizorkin Spaces  with Order Bigger Than 2}\label{s3.1}
\hskip\parindent
In this subsection, we consider the higher order counterparts of Theorems \ref{t-bf}, \ref{apcF2}, \ref{apcF0}, \ref{apcF1} and \ref{apcF} through \ref{t1-xxx},
namely, the corresponding characterizations of Triebel-Lizorkin spaces $ F^{\al}_{p,\,q}(\rn)$
with $\ell\in\mathbb{N},\,p\in(1,\fz),\,q\in(1,\fz]$ and $\al\in(0,2\ell)$. For this purpose, we need to replace the
average operator $B_t$ by its higher order variants. For all $\ell\in\nn$, $t\in(0,\fz)$ and $x\in\rn$, define  the
\emph{$2\ell$-th order average operator   $B_{\ell,t}$}
 by setting, for all $f\in L_\loc^1(\rn)$ and $x\in\rn,$
 \begin{equation*}
 B_{\ell,t}f(x) := -\frac{2}{\binom{2\ell}{\ell} }
 \sum_{j=1}^\ell (-1)^j \binom {2\ell}{\ell-j} B_{jt} f(x),
 \end{equation*}
here and hereafter, $\binom {2\ell}{\ell-j}$ denotes the \emph{binomial coefficients}. Obviously, $B_{1,t}f=B_tf$. Moreover,
 $$( B_{\ell, t} f)(x) =\f {-2} {\binom{2\ell}{\ell}} \sum_{j=1}^\ell (-1)^j \binom{2\ell}{\ell-j}
 ( f * I_{jt})(x),\quad x\in\rn,\ \ t\in(0,\fz).$$

 If we replace the
average operator $B_t$ by $B_{\ell,t}$ in Theorems \ref{t-bf}, \ref{apcF2}, \ref{apcF0}, \ref{apcF1}
and \ref{apcF} through \ref{t1-xxx}, then, by \cite[Lemmas 2.1 and 2.2]{dgyy0}, we have the following theorem, whose proof is similar to the corresponding part of Theorems \ref{t-bf}, \ref{apcF2}, \ref{apcF0}, \ref{apcF1}
and \ref{apcF} through \ref{t1-xxx}, the details being omitted.

\begin{theorem}\label{tt3}
Let $\ell\in\mathbb{N},\,p\in(1,\fz),\,q\in(1,\fz],\,t\in(0,1)$ and $\al\in(0,2\ell)$. Then the conclusions of
Theorems \ref{t-bf} and \ref{apcF2}, (ii) of Theorems \ref{apcF0} and \ref{apcF1}, and the statements b) of Theorems
\ref{apcF} through \ref{t1-xxx} remain hold true when $B_t$ is replaced by $B_{\ell,t}$.
\end{theorem}

\subsection{Triebel-Lizorkin Spaces on Spaces of Homogeneous Type}\label{s3.2}
\hskip\parindent
In this subsection, $(X,\,\rho,\,\mu)$ always denotes
a metric measure space of homogeneous type. Recall that a \emph{quasi-metric} on a nonempty set
$X$ is a non-negative function $\rho$ on $X\times X$ which satisfies

(i) for any $x,\,y\in X,\ \rho(x,y)=0$ if and only if $x=y$;

(ii) for any $x,\,y\in X,\ \rho(x,y)=\rho(y,x)$;

(iii) there exists a constant $K\in[1,\fz)$ such that, for any $x,\,y,\,z\in X$,
\begin{eqnarray}\label{te1}
\rho(x,y)\leq K\lf[\rho(x,z)+\rho(z,y)\r].
\end{eqnarray}
Let $\rho$ be a quasi-metric on $X$, a triple $(X,\rho,\mu)$ is called a \emph{space of homogeneous type} in the sense of Coifman and Weiss \cite{cw71,cw77} if $\mu$ is a regular Borel measure satisfying
the following doubling condition, that is, there exists a constant $\widetilde{C}\in[1,\fz)$ such that,
for all $r\in(0,\fz)$ and $x\in X$,
\begin{eqnarray}\label{te2}
\mu(B(x,2r))\leq \widetilde{C}\mu(B(x,r)),
\end{eqnarray}
where, for any given $r\in(0,\fz)$ and $x\in X$, let
$$B(x,r):=\lf\{y\in X:\ \rho(x,y)<r\r\}$$
be the quasi-metric ball related to $\rho$ of radius $r$ and centering at $x$.

The triple $(X,\rho,\mu)$
is called a \emph{metric measure space of homogeneous type} if $K=1$ in \eqref{te1}
in the above definition of the space of homogeneous type.

Clearly, if $\mu$ is doubling, then, for any $\gamma\in(0,\fz)$, there exists a positive
constant $C_\ga$, which depends on $\ga$ and $\widetilde{C}$ in \eqref{te2}, such that, for all
$r\in(0,\fz)$ and $x\in X$,
\begin{eqnarray*}
\mu(B(x,\ga r))\leq C_{\ga}\mu(B(x,r)).
\end{eqnarray*}

For all $x\in X$ and $t\in(0,\fz)$, let $B(x,t)$ denote a ball with center at
$x$ and radius $t$, and
$\dashint_{B(x,t)} f(y) \,d\mu(y)$ denote the \emph{integral average} of
$f\in L^1_\loc(X)$ on the ball $B(x,t)\subset X$, that is,
\begin{equation*}
B_tf(x):=\dashint_{B(x,t)} f(y) \,d\mu(y):= \frac1{\mu(B(x,t))}\int_{B(x,t)} f(y) \,d\mu(y).
\end{equation*}
Then we have the following conclusions.

\begin{theorem}\label{tt1}
Let $\al\in(0,2),\,p\in(1,\fz)$ and $f\in L^1_\loc(X)$. The following statements are equivalent:

{\rm(i)} there exist
a non-negative $g\in L^p(X)$ and positive constants
$c$, $C$, $\widetilde{C}$ such that, for all $t\in(0,\fz)$ and almost every $x\in X$ and
$y\in B(x,\,ct)$,
\begin{equation*}
\sup_{z\in B(x,t)}\lf|f(z)-B_{Ct} f(z)\r|\le \widetilde{C} t^\al g(y);
\end{equation*}

{\rm(ii)}  there exist a non-negative $g\in L^p(X)$ and positive constants
$c$, $C$, $\widetilde{C}$ such that, for all $t\in(0,\fz)$ and almost every $x\in X$,
\begin{eqnarray*}
\sup_{y\in B(x,t)}\lf|f(y)-B_{Ct} f(y)\r|\le \widetilde{C}t^\al \dashint_{B(x,ct)} g(y)\,d\mu(y);
\end{eqnarray*}

{\rm(iii)} there exist $q\in[1,p)$,
a non-negative $g\in L^p(X)$ and positive constants
$c$, $C$, $\widetilde{C}$ such that, for all $t\in(0,\fz)$ and almost every $x\in X$,
\begin{equation*}
\sup_{y\in B(x,t)}\lf|f(y)-B_{ C t} f(y)\r|\le\widetilde{C} t^\al \lf\{\dashint_{B(x,ct)} [g(y)]^q\,d\mu(y)\r\}^{1/q};
\end{equation*}

{\rm (iv)} there exist a non-negative $g\in L^p(X)$ and positive constants $c,\,C$, $\widetilde{C}$
such that, for all
$t\in(0,\fz)$ and almost every $x\in X$ and $y\in B(x,\,ct)$,
\begin{eqnarray*}
\lf|f(x)-B_{Ct}f(x)\r|\leq \widetilde{C}t^{\az}g(y);
\end{eqnarray*}

{\rm(v)} there exist a non-negative $g\in L^p(X)$ and positive constants
$c$, $C$, $\widetilde{C}$ such that, for all $t\in(0,\fz)$ and almost every $x\in X$,
\begin{equation*}
|f(x)-B_{Ct} f(x)|\le \wz C t^\al \dashint_{B(x,ct)} g(y)\,d\mu(y);
\end{equation*}

{\rm(vi)} there exist $q\in[1,p)$, a non-negative $g\in L^p(X)$ and positive constants
$c$, $C$, $\widetilde{C}$ such that, for all $t\in(0,\fz)$ and almost every $x\in X$,
\begin{equation*}
|f(x)-B_{Ct} f(x)|\le \wz C t^\al \lf\{\dashint_{B(x,ct)} [g(y)]^q\,d\mu(y)\r\}^{1/q}.
\end{equation*}
\end{theorem}

\begin{theorem}\label{tt2}
Let $\al\in(0,2),\,p\in(1,\fz),\,r\in[1,\fz)$ and $f\in L^1_\loc(X)$. The following statements are equivalent:

{\rm(i)} there exist a non-negative $g\in L^p(X)$ and positive constants
$c$, $C$, $\widetilde{C}$ such that, for all $t\in(0,\fz)$ and almost every $x\in X$,
\begin{equation*}
\lf[\dashint_{B(x,t)} |f(y)-B_{Ct} f(y)|^r\,d\mu(y)\r]^{\frac1r}\le \widetilde{C} t^\al \dashint_{B(x,ct)} g(y)\,d\mu(y);
\end{equation*}

{\rm(ii)} there exist $q\in[1,p)$,
a non-negative $g\in L^p(X)$ and positive constants
$c$, $C$, $\widetilde{C}$ such that, for all $t\in(0,\fz)$ and almost every $x\in X$,
\begin{equation*}
\lf[\dashint_{B(x,t)} |f(y)-B_{Ct} f(y)|^r\,d\mu(y)\r]^{\frac1r}\le \widetilde{C} t^\al \lf\{\dashint_{B(x,ct)} [g(y)]^q\,d\mu(y)\r\}^{1/q};
\end{equation*}

{\rm(iii)} there exist a non-negative $g\in L^p(X)$ and positive constants
$C$, $\widetilde{C}$ such that, for all $t\in(0,\fz)$ and almost every $x\in X$,
\begin{equation*}
\lf[\dashint_{B(x,t)} |f(y)-B_{Ct} f(y)|^r\,d\mu(y)\r]^{\frac1r}\le \widetilde{C}t^\al  g(x);
\end{equation*}

{\rm(iv)} there exist
a non-negative $g\in L^p(X)$ and positive constants
$c$, $C$, $\widetilde{C}$ such that, for all $t\in(0,\fz)$ and almost every $x\in X$ and
$y\in B(x,\,ct)$,
\begin{equation*}
\lf[\dashint_{B(x,t)} |f(z)-B_{Ct} f(z)|^r\,d\mu(z)\r]^{\frac1r}\le \widetilde{C} t^\al g(y).
\end{equation*}
\end{theorem}

The proofs of Theorems \ref{tt1} and \ref{tt2} are similar to
those of \cite[Theorems 3.5 and 3.6]{dgyy}, respectively, the details being omitted.

\begin{remark} It would be very interesting to establish the equivalence
between the items of Theorem \ref{tt1} and those of Theorem \ref{tt2}.
Indeed, it is easy to see that Theorem \ref{tt1}(i) implies
Theorem \ref{tt2}(iv). This means that the items of
Theorem \ref{tt1} imply those of Theorem \ref{tt2}.
It is still unknown whether the items of Theorem \ref{tt2}
imply those of Theorem \ref{tt1} or not.

\end{remark}

\bigskip

\noindent Der-Chen Chang

\medskip

\noindent Department of Mathematics and Department of Computer
Science, Georgetown University, Washington D. C. 20057, U. S. A.

\smallskip

\noindent \&

\smallskip

\noindent Department of Mathematics, Fu Jen Catholic University, Taipei 242, Taiwan, China

\smallskip

\noindent{\it E-mail:} \texttt{chang@georgetown.edu}

\bigskip

\noindent  Jun Liu, Dachun Yang (Corresponding author) and Wen Yuan

\medskip

\noindent  School of Mathematical Sciences, Beijing Normal University,
Laboratory of Mathematics and Complex Systems, Ministry of
Education, Beijing 100875, People's Republic of China

\smallskip

\noindent {\it E-mails}: \texttt{junliu@mail.bnu.edu.cn} (J. Liu)

\hspace{1.1cm}\texttt{dcyang@bnu.edu.cn} (D. Yang)

\hspace{1.1cm}\texttt{wenyuan@bnu.edu.cn} (W. Yuan)

\end{document}